\documentclass[11pt]{article}
\usepackage{graphicx}
\usepackage[utf8]{inputenc}
\usepackage[T1]{fontenc}
\usepackage{lmodern}
\usepackage[a4paper]{geometry}
\usepackage{amsmath, amssymb}
\usepackage{amsthm}
\usepackage{thmtools}
\usepackage{mathrsfs}
\usepackage{extarrows}
\usepackage{enumerate}
\usepackage[colorlinks=true]{hyperref}
\theoremstyle{plain}
\newtheorem{thm}{Theorem}[section]
\newtheorem{cor}[thm]{Corollary}
\newtheorem{defi}[thm]{Definition}
\newtheorem{Lem}[thm]{Lemma}
\newtheorem{prop}[thm]{Proposition}

\newtheorem{rem}[thm]{Remark}

\newcommand{\N}{\mathbb N}
\newcommand{\R}{\mathbb R}
\newcommand{\Z}{\mathbb Z}

\renewcommand{\P}{\mathbb P}
\newcommand{\E}{\mathbb{E}}
\renewcommand{\L}{\mathcal L}
\renewcommand{\l}{\mathscr L}
\newcommand{\C}{\mathscr{C}}
\renewcommand{\S}{\mathcal{S}}
\newcommand{\F}{\mathcal{F}}
\newcommand{\ind}{\mathbf{1}}
\newcommand{\dbc}[1]{[\![{#1}]\!]}
\begin{document}

\title{Existence condition of strong stationary times for continuous time Markov chains on discrete graphs
}

\author{Guillaume Copros
}

\date{July 2016}

\maketitle

\begin{abstract}
We consider a random walk on a discrete connected graph having some infinite branches plus finitely many vertices with finite degrees. We find the generator of a strong stationary dual in the sense of Fill, and use it to find some equivalent condition to the existence of a strong stationary time. This strong stationary dual process lies in the set of connected compact sets of the compactification of the graph. When this graph is $\Z$, this is simply the set of (possibly infinite) segments of $\Z$.

\end{abstract}

\pagebreak
\section{Introduction}\label{intro}
	
	\subsection{Historical background}\label{sec-hist}
	Convergence of Markov processes to their stationary distribution has been much studied by analytical tools. For instance, classical theorems provide ways to find a deterministic time $t_0$ at which the distribution of a process $X$ will be arbitrarily close to its stationary distribution $\mu_{\infty}$, for the total variation distance.
	
	A more probabilistic approach is to look for a random time $T<\infty$ such that $X_T$ is distributed \emph{exactly} as its stationary distribution, independently of $T$.
	Such a $T$ is called a \emph{strong stationary time}, and will be defined more precisely in (\ref{def-sst}). This tool was first introduced and developed by Aldous and Diaconis \cite{aldous-diaco} for discrete time Markov chains with finite state space. In this framework, Diaconis and Fill \cite{Diaco} then introduced a method to construct strong stationary times by intertwining processes. This work was transposed to continuous time Markov chains by Fill \cite{fill} and \cite{fill-tts}.
	
	More recently, strong stationary times have been investigated, among others, by Lyzinski and Fill \cite{Lyzinski}, Miclo \cite{Miclo} Gong, Mao and Zhang \cite{Mao} in continuous time, and Lorek and Szekli \cite{Lorek} in discrete time.
	
	They have also been used by Diaconis and Saloff-Coste \cite{Saloff-Coste} to study cut-off phenomenon, Fill \cite{Fill-sampling} for perfect sampling or Fill and Kahn \cite{Kahn} for fastest mixing.
	
	A question that naturally arises is: does it always exist a finite strong stationary time? On a finite state space, Diaconis and Fill \cite{Diaco} proved that the answer is yes. Unfortunately, this is no longer true when the state space is infinite, even countable. Of course, the existence or not depends on the initial distribution: when it is the same as the stationary one, $0$ is an obvious strong stationary time! One can therefore wonder under which conditions there exists a finite one, for a given process, whatever the initial distribution. Thus, Miclo \cite{Miclo} produced some existence conditions in the case of a diffusion on $\R$. We aim here to give similar conditions for continuous time random walks on a connected graph where all vertices but a finite number have degree two.
	
	In Section \ref{sec-preliminaire} we adapt an important result from Fill \cite{fill} to our case, that we will apply in Section \ref{sec-marche-Z} to birth-death processes on $\Z$, and in Section \ref{sec-marche-graphe} to random walks on a graph.

	\subsection{Some reminders on Markov processes}\label{sec-markov}
	What follows recalls briefly the characterization of continuous time Markov chains in terms of generators and semi-groups. For more details, we refer to Norris \cite{norris}. The stochastic processes will be defined on a probability space $(\Omega,\mathcal{A},\P)$, implied in the sequel, which we assume to be "large enough". For a Markov process $(X_t)_{t\geq0}$ on $(\Omega,\mathcal{A},\P)$, we will denote by $X_{[0,t]}$ its trajectory until the time $t$ ($t\geq 0$), $(\mathcal{F}^X_t)_{t\geq0}$ the filtration generated by $X$ and $\mathcal{F}^X_{\infty}$ the $\sigma$-algebra generated by the union of $\mathcal{F}^X_{t}$. If $\tau$ is a $(\mathcal{F}^X_t)$-stopping time, we write:
			$$(\mathcal{F}^X_{\tau}):=\{A\in\mathcal{F}^X_{\infty}:\;A\cap\{\tau\leq t\}\in\mathcal{F}^X_t,\ \forall t\geq0\}$$
	
	A random time $T$ is called a \emph{randomized stopping time} relative to $X$ if there exists a sub $\sigma$-algebra $\mathcal{G}$ of $\mathcal{A}$, independent from $\F^X_{\infty}$, such that $T$ is a stopping time with respect to $(\sigma(\F^X_t,\mathcal{G}))_{t\geq0}$ (cf Fill \cite{fill-tts}, Section 2.2).
	
	Let $\S$ be a countable topological space (not necessarily discrete). We add to this space an isolated point $\Delta$, called \emph{cemetary point}, and we extend every real function $f$ on $\S$ by setting $f(\Delta)=0$. A \emph{stable and conservative Q-matrix} on $\S$ is a matrix $(\L_{x,y})_{(x,y)\in\S^2}$ satisfying:
		$$0\leq\L_{x,y}<+\infty, \quad \forall x\neq y\in\mathcal{S}$$
		$$\sum_{y\neq x}\L_{x,y}=-\L_{x,x}<+\infty, \quad \forall x\in\mathcal{S}$$
	For every Q-matrix $\L$, we will write $\L_x:=-\L_{x,x}$, and we enlarge $\L$ to $\S\cup\{\Delta\}$ by setting $\L_{x,\Delta}=\L_{\Delta,x}=\L_{\Delta,\Delta}=0$ for every $x\in\S$. Together with a probability measure $\mu_0$ on $\S$, a Q-matrix $\L$ defines a process $X$ with values in $\S\cup\{\Delta\}$, through its jump chain $(Y_n)_{n\in\N}$ and jumping times $(T_n)_{n\in\N}$:

		\begin{itemize}
		\item $(Y_n)_{n\in\N}$ is a discrete time Markov chain,  with initial distribution $\mu_0$ and transition matrix $\left((1-\delta_{x,y})\dfrac{\L_{x,y}}{\L_x}\right)_{(x,y)\in\S^2}$, where $\delta_{x,y}$ is the Kronecker delta, equal to $1$ if $x=y$ and $0$ otherwise.
		\item $T_0=0$, and for every $n\in\N^*$, given $Y_0,\dots,Y_n$, the random variables $T_1-T_0,\dots,T_n-T_{n-1}$ are independent and exponentially distributed with respective parameters $\L_{Y_0},\dots,\L_{Y_{n-1}}$.
		\item For all $i\in\N$ and all $t\in\left[T_i,T_{i+1}\right[$, we set $X_t=Y_i$.
		\item Define $T:=\underset{n\rightarrow\infty}{\lim}T_n$. If $T<\infty$, we set $X_t=\Delta$ for every $t\geq T$.
		\end{itemize}

	\noindent This process is Markovian, homogeneous, and its trajectories are right continuous with left limits (càdlàg). In addition, this construction uniquely determines the law of $X$. In the following, we will simply say that $\L$ is a generator on $\S$, and that $X$ is a \emph{minimal process} with generator $\L$ and initial distribution $\mu_0$, or only with generator $\L$ whenever the initial distribution does not matter. The time $T$ is called the \emph{explosion time} of $X$, and we will see in Section \ref{sec-explo} how in certain cases we can naturally construct a non-minimal process, that is, one that is not "killed" after the explosion time. In the case where $T=\infty$ almost surely for every initial distribution, the generator $\L$ (or the process $X$) is said to be nonexplosive, and explosive otherwise. It is positive recurrent (respectively irreducible, reversible) if the jump chain $(Y_n)_{n\in\N}$ is so, for every initial distribution.

	The law of such a process is also determined by its finite dimensional distributions, namely the set of
		$$\P(X_{t_1}=x_1,\cdots,X_{t_n}=x_n)$$
	with $n\in\N,\;0\leq t_1<\dots<t_n<\infty$ and $x_1,\dots,x_n\in\S$. The \emph{transition function} of $X$ is defined by:
		$$P_{x,y}(t)=\P(X_t=y\mid X_0=x)$$
	and is then (together with an initial distribution $\mu_0$) another way to characterize the law of $X$, through its finite dimensional distributions. A transition function defines a \emph{semi-group} $(P(t))_{t\geq0}$ of sub-stochastic matrices, which may be enlarged to $\S\cup\{\Delta\}$ to become stochastic.

	A semi-group $P(\cdot)$ is a \emph{solution of a Kolmogorov equation} associated to $\L$ if for every $t\geq 0$ one of the matrix equalities:

		\begin{gather}
		P'(t)=\L P(t)\quad\quad\text{(backward equation)}\\
		P'(t)=P(t)\L\quad\quad\text{(forward equation)}
		\end{gather}
	
	\noindent hold. It is the minimal solution of these equations if for every other solution $\tilde{P}(\cdot)$ we have $P_{x,y}(t)\leq \tilde{P}_{x,y}(t)$, for every $x,y\in\S$ and every $t\geq0$. We recall that the minimal solution exists (theorem 2.8.3 and 2.8.6 of \cite{norris}), and is necessarily unique by definition.

	Each of the three following conditions is implied by the other two:

		\begin{itemize}
		\item $X$ is a minimal process with generator $\L$
		\item $X$ is a minimal process with semi-group $P(\cdot)$
		\item $P(\cdot)$ is the minimal solution of the backward (respectively forward) equation associated to $\L$
		\end{itemize}
		
	Finally, although we will only use the characterization of processes through $Q$-matrices and transition functions, we also recall a useful consequence in the functional point of view. Let $\mathcal{B}(\S)$ be the set of real valued, bounded functions on $\mathcal{S}$. We see a function $f\in\mathcal{B}(\S)$ as a column vector indexed by $\S$, and in that case $P(\cdot)$ and $\L$ act on $\mathcal{B}(\S)$ by matrix multiplication.

	 If $X$ is a minimal process with generator $\L$, and $\tau_n,\;n\in\N$ are its jumping times, then for all $f\in\mathcal{B}(\S)$ and all $n\in\N$, the process:
		$$f(X_{t\wedge \tau_n})-f(X_0)-\int_{0}^{t\wedge \tau_n}\L[f](X_s)\mathrm{d}s$$
	is a martingale with respect to the filtration $\F^X$. This can be proved by induction on $n$, using the Markov property and the fact that:
		$$\L f(X_s)=\sum\limits_{y\ne X_{\tau_n}}\L_{X_{\tau_n},y}f(y),\quad\forall s\in\left]\tau_n,\tau_{n+1}\right[$$

	\subsection{Strong stationary times and duality}\label{sec-sst}
	We will now deal more specifically with the theory of strong stationary times, by recalling the main definitions and results of Fill \cite{fill-tts} that we will use here. For our purpose, these definitions are given in continuous time, but are easily transposable to the original discrete time case (see e.g. Diaconis and Fill \cite{Diaco}).

	Let $(X_t)_{t\geq0}$ be a Markov chain on a countable set $\S$, positive recurrent, irreducible and nonexplosive, $\mu_t$ its distribution at time $t\geq0$ and $\mu_{\infty}$ its stationary distribution. A \emph{strong stationary time} $T$ of $X$ is a randomized stopping time relative to $X$ satisfying:
		\begin{equation}
		\l(X_T\mid T)=\mu_\infty\quad\text{on the set }\{T<\infty\}\label{def-sst}
		\end{equation}
	(the notation $\l(\cdot)$ stands for the distribution). If in addition $T<+\infty$ almost surely, we say that $T$ is a \emph{finite} strong stationary time.
	
	This tool provides an upper bound for the \emph{separation} between $\mu_t$ and $\mu_{\infty}$, defined by:
		$$\mathfrak{s}(t):=\underset{m\in\S}{\sup}\left(1-\dfrac{\mu_t(m)}{\mu_{\infty}(m)}\right)$$
	since, for every $t\geq0$:
		$$\mathfrak{s}(t)\leq\P(T>t)$$
	When the above inequality is an equality, $T$ is called a \emph{time to stationarity}. We recall that the separation itself is an upper bound on total variation distance, so that strong stationary times can be used to study the rate of convergence of processes.
	
	In practice, to construct a strong stationary time, one will often use (and always in this article) the notion of \emph{strong stationary dual}, defined in continuous time by Fill \cite{fill}, Section 2.1. A process $(X^*_t)_{t\geq0}$, defined on the same probability space $(\Omega,\mathcal{A},\P)$ as $X$, with countable state space $\S^*$ and absorbing state $\infty$, is a strong stationary dual of $X$ if it satisfies, for all $t\geq0$:

		\begin{equation}
		\label{A1}\l(X^*_t\mid \F^X_{\infty})=\l(X^*_t\mid \F^X_{t})
		\end{equation}
		\begin{equation}
		\label{A2}\l(X_t\mid\F^{X^*}_t)=\mu_{\infty}\quad\text{on the set }\{X^*_t=\infty\}
		\end{equation}
		
	\noindent Relation (\ref{A1}) states that $X$ and $X^*$ are adapted to a same filtration. One can show that the time to absorption of a strong stationary dual is a strong stationary time for $X$. Conversely, from a strong stationary time, one can construct a strong stationary dual in a canonical way (see Fill \cite{fill}, Theorem 1). When the time to absorption of the dual is a time to stationarity, the dual is called \emph{sharp}.

	If $\Lambda$ is a transition kernel from $\mathcal{S}^*$ to $\mathcal{S}$, namely:
		$$\forall x^*\in\mathcal{S}^*,\quad \sum_{x\in\mathcal{S}}\Lambda(x^*,x)=\sum_{x\in\mathcal{S}}|\Lambda(x^*,x)|=1$$
	and $X^*$ is a Markov process with values in $\S^*$, we say that $X$ and $X^*$ are $\Lambda$-\emph{linked} or \emph{intertwined by} $\Lambda$ if:
		\begin{equation}\label{A3}
		\quad\quad\l(X_t\mid\F^{X^*}_t)=\Lambda(X^*_t,\cdot)
		\end{equation}
	Fill \cite{fill} showed that, in the framework of his \emph{General settings} (that is, $X$ and $X^*$ are nonexplosive, and the sets $\{x\in\S: \Lambda(x^*,x)>0\}$ are finite for all $x^*\in\S^*$, except maybe the absorbing state $\infty$), such a coupling always exists as soon as the generators $\L$ and $\L^*$ of $X$ and $X^*$ and their initial distribution $\mu_0$ and $\mu_0^*$ respectively are in \emph{algebraic duality}:

		\begin{equation*}
			\begin{cases}
			\mu_0^*\Lambda=\mu_0\\
			\L^*\Lambda=\Lambda\L
			\end{cases}
		\end{equation*}
	
	In Section \ref{sec-preliminaire}, we will extend this result to the framework of reversible random walks on countable graphs having finitely many vertices of degree strictly larger than 2.

	We consider a discrete graph $G$, simple, connected and undirected. We call \emph{infinite branches} of $G$ its maximal (for set inclusion) infinite connected subsets whose points have degree $2$. A more accurate definition is given in Section \ref{sec-marche-graphe}. We assume then that $G$ is formed by a finite number of infinite branches, plus possibly a finite number of points with finite degrees. If there are only infinite branches, then there is only one and $G=\Z$. This particular case is considered in Section \ref{sec-marche-Z}. Otherwise, each of the $N$ infinite branches (denoted by $Q_i,\;i\in I:=\dbc{1,N}$) is isomorphic as a graph to $\N$, and we denote $\varphi_i\colon\N\rightarrow G$ the corresponding isomorphisms. Such a graph is presented in Figure \ref{graphe}.
	
	Next, assume we are given a generator $\L$ on $G$, such that for all $x\neq y\in G,\ \L_{x,y}\neq0$ if and only if $x$ and $y$ are adjacent. Since the graph structure of $G$ will appear only through the generator, we can also choose to consider \emph{a priori} a suitable generator on a countable set, and to endow this latter with a suitable graph structure. That is what will be done in Section \ref{sec-marche-graphe}. We assume that $\L$ is positive recurrent and nonexplosive. This results in the following conditions:

		\begin{equation}
		\sum_{n\in\N}\mu^i(n)<\infty,\quad\quad\forall i\in\dbc{1,N}
		\end{equation}
		\begin{equation}
		\sum_{n\in\N}\dfrac{1}{\mu^i(n)\L_{\varphi_i(n),\varphi_i(n+1)}}\sum_{m=0}^{n}\mu^i(m)=\infty,\quad\quad\forall i\in\dbc{1,N}
		\end{equation}
	
	where:
		\begin{equation}\label{def-mu^i}
		\mu^i(n):=\prod_{k=0}^{n-1}\dfrac{\L_{\varphi_i(k),\varphi_i(k+1)}}{\L_{\varphi_i(k+1),\varphi_i(k)}}
		\end{equation}
	Under these assumptions, there exists a unique stationary distribution $\mu$, and every process $X$ with generator $\L$ converges weakly to $\mu$. We assume in addition that the diagonal of $\L$ is $\mu$-integrable, and that $\L$ is reversible, so that $\mu$ satisfies:
		
		\begin{equation}\label{eq-reversible}
		\forall p,q\in G,\quad \mu(p)\L_{p,q}=\mu(q)\L_{q,p}
		\end{equation}
	
	In this context, we want to study the strong stationary times of $X$, and in particular to determine whether there exists a finite one for any initial distribution. To do so, we will construct in a classical way a dual process with values in $\mathcal{P}(G)$ the set of subsets of $G$, intertwined with $X$ by the kernel $\Lambda$ defined by:
		$$\Lambda(Q,\cdot)=\mu_{\vert_Q}\quad\quad\forall Q\in\mathcal{P}(G)$$
	In the idea of what was done by Diaconis and Fill \cite{Diaco} on finite set or by Miclo \cite{Miclo} or Fill and Lyzinski \cite{Lyzinski} for diffusions on the real line, this dual process can grow at each step by adding adjacent points. The main problem is that it can also lose any boundary point. To keep it from separating in several parts, when this happens we force it to choose one of its connected components, with a probability proportional to its weight relative to $\mu$. We finally get a connected dual process. By studying its conditions of explosion, and showing that for some initial distributions on $X$ there exists a finite strong stationary time if and only if the one provided by this dual process is finite, we get the main result:

	\begin{thm}\label{thm}
	The process $X$ has a finite strong stationary time, whatever the initial distribution $\mu_0$, if and only if for every $i\in I$:

		\begin{equation}\label{borne-infini}
		\sum_{j=1}^{\infty}\mu^i(j+1)\sum_{k=1}^{j}\dfrac{1}{\mu^i(k)\L_{\varphi_i(k),\varphi_i(k+1)}}<\infty
		\end{equation}

	\end{thm}
	
	Miclo \cite{Miclo-resolvent} gave some other equivalent condition to this: considering the restriction $\L_{\F}$ of $\L$ on the set:
		$$\F:=\{f\in\mathbb{L}^2(\mu): \mu[f]=0\},$$
	then (\ref{borne-infini}) holds for all $i\in I$ if and only if $\L_{\F}^{-1}$ is of trace class.
	
	\begin{rem}
	Unlike what was done for diffusions for instance, we do not prove here that the dual process is sharp for some initial distribution. Actually, we conjecture that there would not exist any dual process that would be  $\Lambda$-linked and sharp at the same time as soon as the number of infinite branches is greater than two (except of course in the trivial case where the initial distribution is the stationary one).
	\end{rem}

\section{Results on duality}\label{sec-preliminaire}


\paragraph{}In the sequel, we will work with the following objects:
	\begin{itemize}

	\item $\mathcal{S}$ and $\mathcal{S}^*$ are two countable, topological spaces (not necessarily endowed with discrete topology). The elements of those sets will generally be denoted by $x,y,z$ for $\mathcal{S}$ and $x^*,y^*,z^*$ for $\mathcal{S}^*$.

	\item $\Lambda$ is a transition kernel from $\mathcal{S}^*$ to $\mathcal{S}$.
    
	\item $\L$ and $\L^*$ are two generators, respectively on $\mathcal{S}$ and $\mathcal{S}^*$ and we assume that $\L$ is nonexplosive.

	\item $\mathbf{P}(\cdot)$ and $\mathbf{P}^*(\cdot)$ are the minimal solutions of the Kolmogorov equations associated to $\L$ and $\L^*$ respectively.    

	\item $\mu_0$ and $\mu_0^*$ are two probability measures on $\mathcal{S}$ and $\mathcal{S}^*$ respectively.
	
	\item $(X_t)_{t\geq0}$ is a Markov process with generator $\L$ and initial distribution $\mu_0$.
		    
	\end{itemize}

Instead of the hypothesis of finiteness of the sets $\{x:\Lambda(x^*,x)>0\}$ made in the \emph{General Settings} of Fill \cite{fill} (section 2.2B), we will assume that the diagonal of $\L$ is $\mu$-integrable, that is:
	$$\sum_{x\in\S}\Lambda(x^*,x)\L_x<+\infty$$
In particular, for all $f\in\mathcal{B}(\S)$, the matrix product $\Lambda\L f$ is well defined, associative and finite.

	\subsection{Minimal processes}\label{sec-min}


		\begin{thm}\label{cond-dual}
		If $(\mu_0,\L)$ and $(\mu_0^*,\L^*)$ satisfy the double condition:
		
			\begin{equation}\label{dual-gen}
				\begin{cases}
				\mu_0^*\Lambda=\mu_0\\
				\L^*\Lambda=\Lambda\L
				\end{cases}
			\end{equation}
	
		\noindent then there exists a minimal process $(X^*_t)_{t\geq0}$ on $\mathcal{S}^*$  with generator $\L^*$ and initial distribution $\mu_0^*$, satisfying (\ref{A1}) and (\ref{A3}) until its time of explosion $T$, that is, for all $t\geq0$:
		
			\begin{gather}
			\l(X_t\ |\ \mathcal{F}^{X^*}_t,\ T>t)=\Lambda(X^*_t,\cdot)\label{indep-1}\\
			\l(X^*_t\ |\ X,\ T>t)=\l(X^*_t\ |\ \mathcal{F}^X_t,\ T>t)\label{indep-2}
			\end{gather}

		\end{thm}

	\paragraph{}For all $f\in\mathcal{B}(\S)$ and all $x^*\in\S^*$, we have:
		$$\sum_{y^*\in\S^*}\sum_{y\in\S}|\L^*_{x^*,y^*}\Lambda(y^*,y)f(y)|\le \underset{y\in\S}{\sup}\,f(y)\sum_{y^*\in\S^*}|\L^*_{x^*,y^*}|=2\,\underset{y\in\S}{\sup}\,f(y)\L^*_{x^*}$$
	so the matrix product $\L^*\Lambda f$ is always commutative and finite-valued.
	Under condition (\ref{dual-gen}), we set $\Gamma:=\L^*\Lambda=\Lambda\L$ and define a matrix $\bar{\L}$, which will be that of a coupling generator, by:

		\begin{equation}\label{gen-couple}
		\bar{\L}_{\bar{x},\bar{y}}:=
		\left\lbrace\begin{array}{c l}
		-\left(\L_x+\L^*_{x^*}+\dfrac{\Gamma(x^*,x)}{\Lambda(x^*,x)}\right) & \text{, if } \bar{x}=\bar{y}\\
		\L_{x,y} & \text{, if } y\neq x,\ y^*=x^*\smallskip\\
		\dfrac{\L^*_{x^*,y^*}\Lambda(y^*,x)}{\Lambda(x^*,x)} & \text{, if } y=x,\ y^*\neq x^*\smallskip\\
		\dfrac{\L_{x,y}\L^*_{x^*,y^*}\Lambda(y^*,y)}{\Gamma(x^*,y)} & \text{, if } \Lambda(x^*,y)=0,\ \Gamma(x^*,y)>0\\
		0 & \text{, otherwise} \\
		\end{array}\right.
		\end{equation}
		
\noindent 	with:
		$$\bar{x}=(x,x^*),\;\bar{y}=(y,y^*)\in\mathcal{\bar{S}}:=\{(z,z^*)\in\mathcal{S}\times\mathcal{S}^*:\;\Lambda(z^*,z)>0\}$$
	Some remarks on this definition:\\

		\begin{rem}\label{rem-gen}
		
			\begin{enumerate}[(a)]
			\item$\Lambda(x^*,y)=0$ implies both $x\neq y$ and $x^*\neq y^*$, so these five cases are mutually exclusive (and obviously exhaustive).

			\item \label{rq} If $\L^*_{x^*,y^*}$ or $\L_{x,y}>0$ and $\Lambda(x^*,y)=0$, then with the previous remark one checks that $\Gamma(x^*,y)>0$. Since the fourth coefficient in (\ref{gen-couple}) is $0$ as soon as $\L_{x,y}$ or $\L^*_{x^*,y^*}$ is null, we may replace the fourth case by "$\Lambda(x^*,y)=0,\ \L_{x,y}\L_{x^*,y^*}>0$" without changing the definition of the generator (in that case, we still have $\Gamma(x^*,y)>0$, and in the cases which have been added to the fifth one, the coefficient $\bar{\L}_{\bar{x},\bar{y}}$ is $0$).	
%
%

			\end{enumerate}
		
		\end{rem}

	With the previous remarks and a straightforward computation, one easily checks that the coefficients above define a generator on $\bar{\mathcal{S}}$.

		\begin{proof}[Proof of Theorem \ref{cond-dual}]

		The proof is essentially the same as in Fill \cite{fill}, Proposition 4 and section 2.3. Most parts are identical and given here for convenience.
		
		Let the condition (\ref{dual-gen}) be fulfilled.

		In a first time, we will construct $X^*$ coupled with $X$, so that the couple $(X,X^*)$ has generator $\bar{\L}$. Then we will check that $(X,X^*)$ satisfies (\ref{indep-1}) and (\ref{indep-2}).

		\paragraph{Coupling of the two processes:}
		Having observed $X_0=x_0\in\mathcal{S}$, we set for each $x_0^*\in\S^*$:
		
			\begin{equation}
			X^*_0=x^*_0\quad \text{ with probability }\ \dfrac{\mu^*_0(x^*_0)\Lambda(x^*_0,x_0)}{\mu_0(x_0)}
			\end{equation}
			
		The distribution of $(X_0,X^*_0)$ is given by:

			\begin{align}\label{loi-initiale}
			\bar{\mu}_0(x_0,x^*_0)&=\P(X_0=x_0)\P(X^*_0=x^*_0|X_0=x_0)\nonumber\\
			&=\mu_0(x_0)\dfrac{\mu^*_0(x^*_0)\Lambda(x^*_0,x_0)}{\mu_0(x_0)}\nonumber\\
			&=\mu^*_0(x^*_0)\Lambda(x^*_0,x_0)
			\end{align}

		\noindent which is a probability measure on $\bar{\mathcal{S}}$ since:

			\begin{align*}
			\sum_{x_0,x^*_0}\bar{\mu}_0(x_0,x^*_0)&=\sum_{x^*_0}\mu^*_0(x^*_0)\sum_{x_0}\Lambda(x^*_0,x_0)\\
			&=\sum_{x^*_0}\mu^*_0(x^*_0)\\
			&=1
			\end{align*}
	
		In addition, the marginal distribution of $X^*_0$ is $\mu^*_0$.

		\paragraph{}We set $\tau_0=0$, so at this point $X^*$ is constructed up to this time (recall that $X$ is already constructed). We will now construct $X^*$ step by step through times $\tau_k,\;k\in\N$, where $\tau_k$ is the $k$-th transition of the bivariate chain $(X,X^*)$ (in particular, $X^*$ may not jump at time $\tau_k$).
		
		Let $k\geq1$, and assume by induction that $\tau_0,\dots,\tau_{k-1}$ and $X^*_{[0,\tau_{k-1}]}$ are already constructed. We call $\sigma_k$ the time of the first transition of $X$ following $\tau_{k-1}$ (so it might be equal to $\sigma_{k-1}$ if $X$ has not jumped at time $\tau_{k-1}$). Let $\bar{x}=(x,x^*):=(X,X^*)_{\tau_{k-1}}$, and let $\varepsilon_{k-1}$ be an exponential random variable with parameter $\bar{\L}_{\bar{x}}-\L_x$, independent from $(\epsilon_i)_{i\leq k-2}$ and $X$. Then:

			\begin{enumerate}
			\item If $\tau_{k-1}+\varepsilon_{k-1}>\sigma_k$, we set $\tau_k=\sigma_k$, and:
		
				\begin{equation}
				X^*_{\tau_k}=y^*\quad \text{ with probability }\ \dfrac{\bar{\L}_{\bar{x},\bar{y}}}{\L_{x,y}}
				\end{equation}

			\noindent where $y=X_{\tau_k}$ and $\bar{y}=(y,y^*)$.

			\item Otherwise, we set $\tau_k=\tau_{k-1}+\varepsilon_{k-1}$, and:

				\begin{equation}
				X^*_{\tau_k}=y^*\neq x^*\quad \text{ with probability }\ \dfrac{\bar{\L}_{(x,x^*),(x,y^*)}}{\bar{\L}_{\bar{x}}-\L_x}
				\end{equation}

			\end{enumerate}

		These probabilities are well defined, since $\L_{x,y}$ and $\bar{\L}_{\bar{x}}-\L_x$ are obviously different from $0$ if $X_{\tau_k}=y$ or if $\varepsilon_{k-1}<\infty$ respectively. In addition, in the first case, we get as announced $X^*_{\tau_k}=x^*$ a.s. if $\Lambda(x^*,y)\neq0$, and:
			$$X^*_{\tau_k}=y^*\neq x^*\quad \text{ with probability }\ \dfrac{\L^*_{x^*,y^*}\Lambda(y^*,y)}{\Gamma(x^*,y)}$$
		if $\Lambda(x^*,y)=0$, the sum of these probabilities being equal to $1$.

		\paragraph{}Thus we get $\tau_k-\tau_{k-1}=\min(\varepsilon_{k-1},\sigma_k-\tau_{k-1})$. Now, $\varepsilon_{k-1}\sim\exp(\bar{\L}_{\bar{x}}-\L_x)$ and $\sigma_k-\tau_{k-1}\sim\exp(\L_x)$ are independent, so $\tau_k-\tau_{k-1}\sim\exp(\bar{\L}_{\bar{x}})$. In addition, considering the events $A:=\{(X,X^*)_{\tau_{k-1}}=\bar{x}\}$ and $B:=\{(X,X^*)_{\tau_{k}}=\bar{y}\}$, we get the following transition probabilities:

			\begin{itemize}
			
			\item[$\bullet$]If $y=x,\;y^*\neq x^*$:

				\begin{align*}
				\P(B\ |\ A)&=\P(X_{\tau_k}=x\ |\ A)\;\P(X^*_{\tau_k}=y^*\ |\ X_{\tau_k}=x,\ A)\\
				&=\P(\varepsilon_{k-1}<\sigma_k-\tau_{k-1})\;\P(X^*_{\tau_k}=y^*\ |\ \varepsilon_{k-1}<\sigma_k-\tau_{k-1},\ A)\\
				&=\left(\dfrac{\bar{\L}_{\bar{x}}-\L_x}{\bar{\L}_{\bar{x}}}\right)\left(\dfrac{\bar{\L}_{(x,x^*),(x,y^*)}}{\bar{\L}_{\bar{x}}-\L_x}\right)\\
				&=\dfrac{\bar{\L}_{(x,x^*),(x,y^*)}}{\bar{\L}_{\bar{x}}}
				\end{align*}
	
			\item[$\bullet$]If $y\neq x$:

				\begin{align*}
				\P(B\ |\ A)&=\P(X_{\tau_k}=y\ |\ A)\;\P(X^*_{\tau_k}=y^*\ |\ X_{\tau_k}=y,\ A)\\
				&=\dfrac{\L_{x,y}}{\L_x}\;\P(\varepsilon_{k-1}>\sigma_k-\tau_{k-1}\ |\ A)\;\P(X^*_{\tau_k}=y^*\ |\ X_{\tau_k}=y,\ A)\\
				&=\dfrac{\L_{x,y}}{\L_x}\left(\dfrac{\L_x}{\bar{\L}_{\bar{x}}}\right)\left(\dfrac{\bar{\L}_{\bar{x},\bar{y}}}{\L_{x,y}}\right)\\
				&=\dfrac{\bar{\L}_{\bar{x},\bar{y}}}{\bar{\L}_{\bar{x}}}		
				\end{align*}

			\end{itemize}

		By construction, the process $(X,X^*)$ is a Markov jump process, defined up to its first explosion time, and satisfying (\ref{indep-2}). We have proven with the previous computations that its generator is indeed $\bar{\L}$.

		\paragraph{Check of the intertwining (relation (\ref{indep-1})):} This will be made in three steps. Firstly, we will show an equation with the generators:

			\begin{equation}\label{2.39}
			\forall x^*\in\mathcal{S}^*,\;f\in\mathcal{B}(\bar{\S}), \qquad \sum_{x\in\mathcal{S}}\sum_{\bar{y}\in\bar{\S}}\Lambda(x^*,x)\bar{\L}_{\bar{x},\bar{y}}f(\bar{y})=\sum_{\bar{y}\in\bar{\S}}\L^*_{x^*,y^*}\Lambda(y^*,y)f(\bar{y})
			\end{equation}

		\noindent Heuristically, the l.h.s. corresponds to all the transitions from some given $x^*\in\S^*$ to a $x\in\S$, and then from $(x,x^*)\in\bar{\S}$ to $(y,y^*)\in\bar{\S}$, whereas the r.h.s. correspond to a transition from $x^*$ to $y^*$, then to $y$, without wondering about $x$.
		
		In a second time, we will use this equation to prove an inequality involving the corresponding semi-groups:

			\begin{equation}\label{2.41}
			\forall x^*\in\mathcal{S}^*,\;\bar{y}\in\bar{\mathcal{S}}, \quad \sum_{x\in\mathcal{S}}\Lambda(x^*,x)\bar{\mathbf{P}}_{\bar{x},\bar{y}}(t)\geq \mathbf{P}^*_{x^*,y^*}(t)\Lambda(y^*,y)
			\end{equation}

		\noindent and finally, we will deduce by induction on $k$:

			\begin{equation}\label{2.21}
			\P(X^*_{t_0}=x^*_0,\dots,X^*_{t_k}=x^*_k,\;X_{t_k}=x_k)=\mu^*_0(x^*_0)\prod_{i=1}^{k}\mathbf{P}^*_{x^*_{i-1},x^*_i}(t_i-t_{i-1})\Lambda(x^*_k,x_k)
			\end{equation}

		\noindent $\forall k\geq0,\ \forall0=t_0\leq\dots\leq t_k,\ \forall x^*_0,\dots,x^*_k\in\mathcal{S}^*,\ \forall x_k\in\mathcal{S}$, which implies both (\ref{indep-1}) and that $\mathbf{P}^*(\cdot)$ is the semi-group of $X^*$.

		\paragraph{}Proof of (\ref{2.39}): we start by proving the following matrix equality:
			\begin{equation}\label{2.39bis}
			\sum_{x\in\mathcal{S}}\Lambda(x^*,x)\bar{\L}_{\bar{x},\bar{y}}=\L^*_{x^*,y^*}\Lambda(y^*,y)
			\end{equation}
		Let $x^*\in\mathcal{S}^*,\;\bar{y}\in\bar{\mathcal{S}}$, with $x^*\neq y^*$. If $\Lambda(x^*,y)\neq0$, then from the definition of $\bar{\L},\;\bar{\L}_{\bar{x},\bar{y}}\neq0$ implies $y=x$. Hence:

			\begin{align*}
			\sum_{x\in\mathcal{S}}\Lambda(x^*,x)\bar{\L}_{\bar{x},\bar{y}}&=
				\begin{cases}
				\Lambda(x^*,y)\bar{\L}_{(y,x^*),\bar{y}}\quad \text{if }(y,x^*)\in\bar{\mathcal{S}}\smallskip\\
				\underset{x\neq y}{\sum}\Lambda(x^*,x)\bar{\L}_{\bar{x},\bar{y}}\quad \text{otherwise}\\
				\end{cases}\\
				&=\begin{cases}
				\Lambda(x^*,y)\dfrac{\L^*_{x^*,y^*}\Lambda(y^*,y)}{\Lambda(x^*,y)}\quad \text{ if }(y,x^*)\in\bar{\mathcal{S}}\smallskip\\
				\underset{x\neq y}{\sum}\Lambda(x^*,x)\dfrac{\L_{x,y}\L^*_{x^*,y^*}\Lambda(y^*,y)}{\Gamma(x^*,y)} \quad\text{otherwise}
				\end{cases}\\
			&=\L^*_{x^*,y^*}\Lambda(y^*,y)
			\end{align*}

		And if $x^*\in\mathcal{S}^*,\;\bar{y}\in\bar{\mathcal{S}}$, with $x^*=y^*$, then:

			\begin{align*}
			\sum_{x:\bar{x}\in\bar{\mathcal{S}}}\Lambda(x^*,x)\bar{\L}_{\bar{x},\bar{y}}&=-\Lambda(x^*,y)\bar{\L}_{(y,x^*)}+\sum_{x\neq y}\Lambda(x^*,x)\bar{\L}_{\bar{x},\bar{y}}\\
			&=-\Lambda(x^*,y)\left(\L_y+\L^*_{x^*}+\dfrac{\Gamma(x^*,y)}{\Lambda(x^*,y)}\right)+\sum_{x\neq y}\Lambda(x^*,x)\L_{x,y}\\
			&=-\Lambda(x^*,y)(\L_y+\L^*_{x^*})-\Gamma(x^*,y)+\Gamma(x^*,y)+\Lambda(x^*,y)\L_y\\
			&=-\L^*_{x^*}\Lambda(x^*,y)\\
			&=\L^*_{x^*,y^*}\Lambda(y^*,y)
			\end{align*}
			
		Now, for $f\in\mathcal{B}(\bar{\S})$, we have:
		
			\begin{align*}
				&\sum_{x:\bar{x}\in\bar{\mathcal{S}}}\sum_{\bar{y}\in\bar{\S}}\Lambda(x^*,x)\bar{\L}_{\bar{x},\bar{y}}f(\bar{y})=\sum_{x:\bar{x}\in\bar{\mathcal{S}}}\Bigg[\Lambda(x^*,x)\bar{\L}_{\bar{x},\bar{x}}f(\bar{x})+\underset{y:(y,x^*)\in\bar{\S}\atop y\neq x}{\sum}\Lambda(x^*,x)\L_{x,y}f((y,x^*))\\
				&\quad+\underset{y^*:(x,y^*)\in\bar{\S}\atop y^*\neq x^*}{\sum}\L^*_{x^*,y^*}\Lambda(y^*,x)f((x,y^*))+\sum_{\bar{y}:(y,x^*)\notin\bar{\S}}\dfrac{\Lambda(x^*,x)\L_{x,y}\L^*_{x^*,y^*}\Lambda(y^*,y)f(\bar{y})}{\Gamma(x^*,y)}\Bigg]\\
				=&\sum_{x:\bar{x}\in\bar{\mathcal{S}}}\Lambda(x^*,x)\bar{\L}_{\bar{x},\bar{x}}f(\bar{x})+\sum_{x:\bar{x}\in\bar{\mathcal{S}}}\underset{y:(y,x^*)\in\bar{\S}\atop y\neq x}{\sum}\Lambda(x^*,x)\L_{x,y}f((y,x^*))\\
				&\quad+\sum_{x:\bar{x}\in\bar{\mathcal{S}}}\underset{y^*:(x,y^*)\in\bar{\S}\atop y^*\neq x^*}{\sum}\L^*_{x^*,y^*}\Lambda(y^*,x)f((x,y^*))+\sum_{x:\bar{x}\in\bar{\mathcal{S}}}\sum_{\bar{y}:(y,x^*)\notin\bar{\S}}\dfrac{\Lambda(x^*,x)\L_{x,y}\L^*_{x^*,y^*}\Lambda(y^*,y)f(\bar{y})}{\Gamma(x^*,y)}
			\end{align*}
			
		\noindent provided that the four sums in the last equality converge. For the first one, recall that:
			$$\bar{\L}_{\bar{x},\bar{x}}=\L_{x,x}+\L^*_{x^*,x^*}-\dfrac{\Gamma(x^*,x)}{\Lambda(x^*,x)}$$
		so we see that the first sum converges since the diagonal of $\L$ is $\mu$-integrable. The second one converges for the same reason, the third one because $f$ is bounded and:
			$$\sum_{y^*\in\S^*}\sum_{x\in\S}|\L^*_{x^*,y^*}|\Lambda(y^*,x)=2\L^*_{x^*}<+\infty,$$
		and the last one because $f$ is bounded and:
			\begin{align*}
			\sum_{y^*\in\S^*}\sum_{y:(y,x^*)\notin\bar{\S}}\sum_{x\in\S}\dfrac{\Lambda(x^*,x)|\L_{x,y}\L^*_{x^*,y^*}|\Lambda(y^*,y)}{\Gamma(x^*,y)}&=\sum_{y^*\in\S^*}\sum_{y:(y,x^*)\notin\bar{\S}}|\L^*_{x^*,y^*}|\Lambda(y^*,y)\\
			&\le2\L^*_{x^*}
			\end{align*}
		
		Getting back to our previous computaion, we deduce:
			\begin{align*}
				&\sum_{x:\bar{x}\in\bar{\mathcal{S}}}\sum_{\bar{y}\in\bar{\S}}\Lambda(x^*,x)\bar{\L}_{\bar{x},\bar{y}}f(\bar{y})=\sum_{x:\bar{x}\in\bar{\mathcal{S}}}\Lambda(x^*,x)\bar{\L}_{\bar{x},\bar{x}}f(\bar{x})+\sum_{x:\bar{x}\in\bar{\mathcal{S}}}\underset{y:(y,x^*)\in\bar{\S}\atop y\neq x}{\sum}\Lambda(x^*,x)\L_{x,y}f((y,x^*))\\
				&\quad+\sum_{x:\bar{x}\in\bar{\mathcal{S}}}\underset{y^*:(x,y^*)\in\bar{\S}\atop y^*\neq x^*}{\sum}\L^*_{x^*,y^*}\Lambda(y^*,x)f((x,y^*))+\sum_{x:\bar{x}\in\bar{\mathcal{S}}}\sum_{\bar{y}:(y,x^*)\notin\bar{\S}}\dfrac{\Lambda(x^*,x)\L_{x,y}\L^*_{x^*,y^*}\Lambda(y^*,y)f(\bar{y})}{\Gamma(x^*,y)}\\
				=&\sum_{x:\bar{x}\in\bar{\mathcal{S}}}\Lambda(x^*,x)\bar{\L}_{\bar{x},\bar{x}}f(\bar{x})+\sum_{y:(y,x^*)\in\bar{\S}}\sum_{x\neq y\in\mathcal{S}}\Lambda(x^*,x)\L_{x,y}f((y,x^*))\\
				&\quad+\sum_{y^*\in\S^*}\sum_{x:(x,y^*)\in\bar{\S}}\L^*_{x^*,y^*}\Lambda(y^*,x)f((x,y^*))+\sum_{\bar{y}:(y,x^*)\notin\bar{\S}}\sum_{x\neq y\in\S}\dfrac{\Lambda(x^*,x)\L_{x,y}\L^*_{x^*,y^*}\Lambda(y^*,y)f(\bar{y})}{\Gamma(x^*,y)}\\
				&=\sum_{\bar{y}\in\bar{\S}}\sum_{x:\bar{x}\in\bar{\mathcal{S}}}\Lambda(x^*,x)\bar{\L}_{\bar{x},\bar{y}}f(\bar{y})\\
				&=\sum_{\bar{y}\in\bar{\S}}\L^*_{x^*,y^*}\Lambda(y^*,y)f(\bar{y})
			\end{align*}
		
		\noindent by (\ref{2.39bis}).

		\paragraph{}Proof of (\ref{2.41}): Let $\bar{\mathbf{P}}(\cdot)$ be the minimal solution of the Kolmogorov equations associated to $\bar{\L}$. For every $t\geq0$, we define the matrix $(Q_{x^*,\bar{y}}(t))_{(x^*,\bar{y})\in\mathcal{S}^*\times\bar{\mathcal{S}}}$ as follows:

			\begin{equation*}
			Q_{x^*,\bar{y}}(t):=\underset{x:\bar{x}\in\mathcal{\bar{S}}}{\sum}\Lambda(x^*,x)\bar{\mathbf{P}}_{\bar{x},\bar{y}}(t)
			\end{equation*}

		Observe that for every $(x^*,\bar{y})\in\mathcal{S}^*\times\bar{\mathcal{S}}$, the function $Q_{x^*,\bar{y}}(\cdot)$ is the uniform limit of the sequence $(Q^{(n)}_{x^*,\bar{y}}(\cdot))_n$ defined by:
			$$Q^{(n)}_{x^*,\bar{y}}(t):=\underset{x\in A_n}{\sum}\Lambda(x^*,x)\bar{\mathbf{P}}_{\bar{x},\bar{y}}(t)$$
		where $(A_n)_n$ is a sequence of finite subsets of $\S$ such that $A_n\nearrow\S$. Thus:
			\begin{align*}
			Q'_{x^*,\bar{y}}(t)&=\sum_{x:\bar{x}\in\mathcal{\bar{S}}}\Lambda(x^*,x)\bar{\mathbf{P}}'_{\bar{x},\bar{y}}(t)\\
			&=\sum_{x:\bar{x}\in\mathcal{\bar{S}}}\Lambda(x^*,x)\sum_{\bar{z}}\bar{\L}_{\bar{x},\bar{z}}\bar{\mathbf{P}}_{\bar{z},\bar{y}}(t)\\
			&=\sum_{\bar{z}}\L^*_{x^*,z^*}\Lambda(z^*,z)\bar{\mathbf{P}}_{\bar{z},\bar{y}}(t)\\
			&=\sum_{z^*\in\mathcal{S}^*}\L^*_{x^*,z^*}Q_{z^*,\bar{y}}(t)
			\end{align*}

		\paragraph{}We then get, $\forall (x^*,\bar{y})\in\mathcal{S}^*\times\bar{\S}$:

			\begin{align*}
			\nonumber\dfrac{\mathrm{d}}{\mathrm{d}t}\left[\exp(\L^*_{x^*}t)Q_{x^*,\bar{y}}(t)\right]&= \left[Q'_{x^*,\bar{y}}(t)+\L^*_{x^*}Q_{x^*,\bar{y}}(t)\right]\exp(\L^*_{x^*}t)\\
			\nonumber&=\left[\sum_{z^*}\L^*_{x^*,z^*}Q_{z^*,\bar{y}}(t)+\L^*_{x^*}Q_{x^*,\bar{y}}(t)\right]\exp(\L^*_{x^*}t)\\
			&=\sum_{z^*\neq x^*}\L^*_{x^*,z^*}Q_{z^*,\bar{y}}(t)\exp(\L^*_{x^*}t)
			\end{align*}

		This sum is finite so the function is continuous in $t$. We apply the fundamental theorem of calculus,	with $Q_{x^*,\bar{y}}(0)=\delta_{x^*,y^*}\Lambda(y^*,y)$:

			\begin{align}
			\exp(\L^*_{x^*}t)Q_{x^*,\bar{y}}(t)&=\delta_{x^*,y^*}\Lambda(x^*,y)+\int_0^t\sum_{z^*\neq x^*}\L^*_{x^*,z^*}Q_{z^*,\bar{y}}(s)\exp(\L^*_{x^*}s)\mathrm{d}s \nonumber\\
			Q_{x^*,\bar{y}}(t)&=\delta_{x^*,y^*}\exp(-\L^*_{x^*}t)\Lambda(x^*,y)\nonumber\\
			&\quad+\int_0^t\sum_{z^*\neq x^*}\L^*_{x^*,z^*}Q_{z^*,\bar{y}}(s)\exp(\L^*_{x^*}(s-t))\mathrm{d}s 
			\end{align}

		\noindent We have expressed $Q$ as a fixed point for the function $f$ defined by:
			\begin{equation*} 
			f\colon \mathcal{M_{\S^*,\bar{\S}}(\mathcal{C}(\R^+,\;\R^+))}\rightarrow\mathcal{M_{\S^*,\bar{\S}}(\mathcal{C}(\R^+,\;\R^+))}
			\end{equation*}
			\begin{equation}
			f(R)_{x^*,\bar{y}}(t)=\delta_{x^*,y^*}\exp(-\L^*_{x^*}t)\Lambda(x^*,y)
			+\int_0^t\sum_{z^*\neq x^*}\L^*_{x^*,z^*}R_{z^*,\bar{y}}(s)\exp(\L^*_{x^*}(s-t))\mathrm{d}s \label{expr_Q}
			\end{equation}
		so that $Q=f(Q)$.
		
		We call:

			\begin{itemize}
			
			\item $Y^*$ a Markov process with generator $\L^*$
			\item $J$ the jumping times chain of the process $Y^*$

			\item $P_{x^*,y^*}^{*(N)}(t)$ the probability that $Y^*$ goes from $x^*$ to $y^*$ in time $t$, and with a maximum of $N-1$ steps:
				$$P_{x^*,y^*}^{*(N)}(t):=\P_{x^*}(Y^*_t=y^*,t< J_N)$$
	  
			\item $Q^{*(N)}_{x^*,\bar{y}}(t):=P^{*(N)}_{x^*,y^*}(t)\Lambda(y^*,y)$
			
			\end{itemize}

		\noindent Then we have:

			\begin{align*}
			P_{x^*,y^*}^{*(1)}(t)&=\P_{x^*}(t<J_1)\delta_{x^*,y^*}\\
			&=\exp(-\L^*_{x^*}t)\delta_{x^*,y^*}
			\end{align*}

		\noindent and so:
			$$Q_{x^*,\bar{y}}^{*(1)}(t)=\delta_{x^*,y^*}\exp(-\L^*_{x^*}t)\Lambda(x^*,y)\leq Q_{x^*,\bar{y}}(t)$$
		We will show by induction that $Q^{*(N)}=f(Q^{*(N-1)})$, where $f$ is the function defined in (\ref{expr_Q}), to derive the inequality:

			\begin{equation*}
			Q^*\leq Q\quad(\textit{i.e. } Q^*_{x^*,\bar{y}}(t)\leq Q_{x^*,\bar{y}}(t),\ \forall x^*,\bar{y},t)
			\end{equation*}

		\noindent where:
			$$Q^*_{x^*,\bar{y}}(t):=\mathbf{P}^*_{x^*,y^*}(t)\Lambda(y^*,y)$$
		To do so, we condition by $Y^*_0=x^*$ and we separate in two disjoint events:

			\begin{equation*}
			\{Y^*_t=y^*,t< J_N\}=\{Y^*_t=y^*,t< J_1\}\cup\{Y^*_t=y^*,J_1\leq t< J_N\}
			\end{equation*}

		\noindent the first one having probability $P_{x^*,y^*}^{*(1)}(t)$. For the second one, we condition by $J_1$ (which is exponentially distributed):

			\begin{align*}
			\P_{x^*}(Y^*_t=y^*,J_1\leq t< J_N)&=\E_{x^*}(\mathbf{1}_{t<J_N}\mathbf{1}_{Y^*_t=y^*}\mathbf{1}_{J_1<t})\\
			&=\E_{x^*}\left[\E_{x^*}(\mathbf{1}_{t<J_N}\mathbf{1}_{Y^*_t=y^*}\ |\ J_1)\mathbf{1}_{J_1<t}\right]\\
			&=\E_{x^*}\left[\P_{x^*}(t<J_N,Y^*_t=y^*\ |\ J_1)\mathbf{1}_{J_1<t}\right]\\
			\end{align*}

		\noindent then by $Y^*_{J_1}$, which is independent from $J_1$ and is equal to $z^*\neq x^*$ with probability $\dfrac{\L^*_{x^*,z^*}}{\L^*_{x^*}}$, and we use the homogeneity of $Y^*$:

			\begin{align*}
			&\P_{x^*}(Y^*_t=y^*,J_1<t< J_N)\\
			=&\E_{x^*}\left[\P_{x^*}(t<J_N,Y^*_t=y^*\ |\ J_1)\mathbf{1}_{J_1<t}\right]\\
			=&\int_{0}^{t}\L^*_{x^*}\mathrm{e}^{-s\L^*_{x^*}}\P_{x^*}(t<J_N,Y^*_t=y^*\ |\ J_1=s)\mathrm{d}s\\
			=&\int_{0}^{t}\sum_{z^*\in\mathcal{S}^*}\L^*_{x^*}\mathrm{e}^{-s\L^*_{x^*}}\P_{x^*}(t<J_N,Y^*_t=y^*\ |\ J_1=s,Y^*_{J_1}=z^*)\P_{x^*}(Y^*_{J_1}=z^*)\mathrm{d}s\\
			=&\int_{0}^{t}\sum_{z^*\neq x^*}\L^*_{x^*,z^*}\mathrm{e}^{-s\L^*_{x^*}}\P_{z^*}(Y^*_{t-s}=y^*,t-s\leq J_{N-1})\mathrm{d}s\\
			=&\int_{0}^{t}\sum_{z^*\neq x^*}\L^*_{x^*,z^*}\mathrm{e}^{-s\L^*_{x^*}}P_{z^*,y^*}^{*(N-1)}(t-s)\mathrm{d}s\\
			=&\int_{0}^{t}\sum_{z^*\neq x^*}\L^*_{x^*,z^*}\mathrm{e}^{-(t-s)\L^*_{x^*}}P_{z^*,y^*}^{*(N-1)}(s)\mathrm{d}s
			\end{align*}

		Regrouping the terms, we get:
			$$P_{x^*,y^*}^{*(N)}(t)=\exp(-\L^*_{x^*}t)\delta_{x^*,y^*}+\int_{0}^{t}\sum_{z^*\neq x^*}\L^*_{x^*,z^*}\mathrm{e}^{-(t-s)\L^*_{x^*}}P_{z^*,y^*}^{*(N-1)}(s)\mathrm{d}s$$
		so $Q_{x^*,\bar{y}}^{*(N)}(t)=f(Q_{x^*,\bar{y}}^{*(N-1)}(t))$. Since $f$ is non-decreasing, we get by induction $Q_{x^*,\bar{y}}^{*(N)}(t)\leq Q_{x^*,\bar{y}}(t)$. By monotone convergence theorem, $(Q_{x^*,\bar{y}}^{*(N)}(t))_N$ goes to $Q^*_{x^*,\bar{y}}(t)\leq Q_{x^*,\bar{y}}(t)$, which concludes the proof of (\ref{2.41}).
		
		When $\L^*$ is nonexplosive, $Q^*$ is a stochastic matrices, the inequality above implies the equality $Q^*=Q$ (since $Q$ is substochastic) and we can show (\ref{2.21}) by an easy induction on $k$ (see below). Unfortunately, this argument is no longer valid if $\L^*$ is explosive. The following reasoning shows the nonexplosive case first, then use it to show the explosive case. 

		\paragraph{Proof of (\ref{2.21}): }Recall that we want to prove by induction on $k$:
			$$\P(X^*_{t_0}=x^*_0,\dots,X^*_{t_k}=x^*_k,X_{t_k}=x_k)=\mu^*_0(x^*_0)\prod_{i=1}^{k}\mathbf{P}^*_{x^*_{i-1},x^*_i}(t_i-t_{i-1})\Lambda(x^*_k,x_k)$$

		$T$ is the explosion time of $X^*$ and actually that of $(X,X^*)$ since we assume that $X$ is nonexplosive. We recall (\ref{loi-initiale}):
			$$\bar{\mu}_0(x_0,x^*_0)=\mu^*_0(x^*_0)\Lambda(x^*_0,x_0)$$
		which gives us immediately the base case for $k=0$. Suppose now that (\ref{2.21}) is true for some $k\geq0$. Let $x^*_0,\dots,x^*_{k+1}\in\mathcal{S}^*$ and $x_{k+1}\in\mathcal{S}$. Then we have:

			\begin{align*}
			\P(X^*_{t_0}&=x^*_0,\dots,X^*_{t_{k+1}}=x^*_{k+1},X_{t_{k+1}}=x_{k+1})\\
			&=\sum_{x_k\in\mathcal{S}}\P(X^*_{t_0}=x^*_0,\dots,X^*_{t_{k}}=x^*_{k},X_{t_{k}}=x_{k})\\
			&\quad\quad\times\P(X^*_{t_{k+1}}=x^*_{k+1},X_{t_{k+1}}=x_{k+1}\ |\ X^*_{t_0}=x^*_0,\dots,X^*_{t_{k}}=x^*_{k},X_{t_{k}}=x_{k})\\
			&=\mu^*_0(x^*_0)\prod_{i=1}^{k}\mathbf{P}^*_{x^*_{i-1},x^*_i}(t_i-t_{i-1})\sum_{x_k\in\mathcal{S}}\Lambda(x^*_k,x_k)\bar{\mathbf{P}}_{\bar{x}_k,\bar{x}_{k+1}}(t_{k+1}-t_k)\\
			&\geq \mu^*_0(x^*_0)\prod_{i=1}^{k}\mathbf{P}^*_{x^*_{i-1},x^*_i}(t_i-t_{i-1})\mathbf{P}^*_{x^*_k,x^*_{k+1}}(t_{k+1}-t_k)\Lambda(x^*_{k+1},x_{k+1})\\
			&=\mu^*_0(x^*_0)\prod_{i=1}^{k+1}\mathbf{P}^*_{x^*_{i-1},x^*_i}(t_i-t_{i-1})\Lambda(x^*_{k+1},x_{k+1})
			\end{align*}

		\noindent where we have applied successively the law of total probability for the first equality, the Markovianity of $(X,X^*)$ and the induction hypothesis for the second one, and (\ref{2.41}) for the inequality.
	
		We have seen that if $\L^*$ is nonexplosive, the inequality is actually an equality and the induction is over. In particular, this shows that $\mathbf{P}^*(\cdot)$ is the semi-group of the marginal $X^*$ and consequently that $\L^*$ is its generator. The process $X^*$ has been built from the trajectories of $X$, independently from the explosiveness of its generator $\L^*$. Its jump rates are thus given by the coefficients $\L^*_{x^*,y^*},\;x^*,y^*\in\S^*$ and it follows that $\L^*$ is also the generator of the minimal process $X^*$ without the assumption of nonexplosiveness

		Let us assume now that $\L^*$ is explosive. The first and last expressions in the computation above are both positive for any choice of $x^*_i$, $x_{k+1}$. Now, their sums in all possible $x^*_0,\dots,x^*_{k+1},x_{k+1}$ are both equal to $1-\P(T\leq t_{k+1})$ (since $\mathbf{P}^*(\cdot)$ is the minimal semi-group associated to $\L^*$), which shows the equality and ends the induction.

		\end{proof}

		\begin{cor}\label{cor-tps-darret}
		If the condition (\ref{dual-gen}) holds, then the process $(X,X^*)$ also satisfies:

			\begin{equation}
			\l(X_{\tau}\ |\ \mathcal{F}^{X^*}_{\tau},\;\tau<T)=\Lambda(X^*_{\tau},\cdot)
			\end{equation}

		\noindent for every $(\mathcal{F}^{X^*}_t)$-stopping time $\tau$.
		\end{cor}
		
		Here, the inequality $\tau<T$ has to be understood in $\bar{\R}$. In particular, $T$ and $\tau$ might be infinite but the event $\{\tau<T\}$ implies $\{\tau<\infty\}$.

		\begin{proof}
		Let $\tau$ be a $(\mathcal{F}^{X^*}_t)$-stopping time, $f\colon\mathcal{S}\rightarrow\R$ be measurable, continuous and bounded and $G$ be a real-valued random variable, measurable with respect to $\mathcal{F}^{X^*}_{\tau}$ and bounded. By definition of the conditional distribution, we want to show:

			\begin{equation}
			\E(G\ind_{\tau<T}f(X_{\tau}))=\E(G\ind_{\tau<T}\Lambda[f](X^*_{\tau}))
			\end{equation}

		For all $n\in\N$, we set $\tau_n:=\dfrac{1}{n}\lceil n\tau\rceil$, where $\lceil\cdot\rceil$ denote the ceiling function ($\lceil x\rceil$ is the smallest integer greater than or equal to $x$). We get:
			\begin{align*}
			\{\tau_n\le t\}=&\{\lceil n\tau\rceil\le nt\}\\
			=&\underset{m\in\N\atop m\le nt}{\bigcup}\{m=\lceil n\tau\rceil\}\\
			=&\underset{m\in\N\atop m\le nt}{\bigcup}\{m-1<n\tau\leq m\}\\
			=&\underset{m\in\N\atop \frac{m}{n}\le t}{\bigcup}\{\dfrac{m-1}{n}<\tau\leq \dfrac{m}{n}\}\in\F^{X^*}_t
			\end{align*}
		so $\tau_n$ is a $(\F^{X^*}_t)$-stopping time. $\tau_n$ takes values in $\dfrac{1}{n}\N\cup\{\infty\}$, and $\tau_n=k/n$ if and only if $\dfrac{k-1}{n}<\tau\leq\dfrac{k}{n}$, so:
			$$G\ind_{\tau_n=k/n}=G\ind_{\tau\leq k/n}-G\ind_{\tau\leq (k-1)/n}$$
		which is measurable with respect to $\F^{X^*}_{k/n}$.
		
		We recall that (\ref{indep-1}) holds, namely:
			$$\forall t>0,\quad\l(X_t\ |\ \mathcal{F}^{X^*}_t,t<T)=\Lambda(X^*_t,\cdot)$$
		or in other words, by definition, for every function $h\colon\mathcal{S}\rightarrow\R$ measurable, continuous and bounded, every real-valued random variable $H$ measurable with respect to $\mathcal{F}_t^{X^*}$ and bounded, and every $t>0$:
			$$\E(H\ind_{t<T}h(X_t))=\E(H\ind_{t<T}\Lambda[h](X^*_t))$$
		Hence:

			\begin{align*}
			\E(\ind_{\tau_n<T}Gf(X_{\tau_n}))&=\E\left(\sum_{k=0}^{\infty}\ind_{k/n<T}\ind_{\tau_n=k/n}Gf(X_{k/n})\right)\\
			&=\sum_{k=0}^{\infty}\E\left(\ind_{k/n<T}\ind_{\tau_n=k/n}Gf(X_{k/n})\right)\\
			&=\sum_{k=0}^{\infty}\E\left(\ind_{k/n<T}\ind_{\tau_n=k/n}G\Lambda[f](X^*_{k/n})\right)\\
			&=\E\left(\sum_{k=0}^{\infty}\ind_{k/n<T}\ind_{\tau_n=k/n}G\Lambda[f](X^*_{k/n})\right)\\
			&=\E\left(\ind_{\tau_n<T}G\Lambda[f](X^*_{\tau_n})\right)
			\end{align*}

		Now, almost surely $\lim\limits_{n\rightarrow\infty}f(X_{\tau_n})=f(X_\tau)$ by right continuity of the trajectories of $(X,X^*)$ and continuity of $f$, $\lim\limits_{n\rightarrow\infty}\ind_{\tau_n<T}=\ind_{\tau<T}$ and $\lim\limits_{n\rightarrow\infty}\Lambda[f](X^*_{\tau_n})=\Lambda[f](X^*_\tau)$ by continuity of $\Lambda[f]$, so we conclude using dominated convergence:

			\begin{align*}
			\E(G\ind_{\tau<T}f(X_{\tau}))&=\lim\limits_{n\rightarrow\infty}\E\left(\ind_{\tau_n<T}Gf(X_{\tau_n})\right)\\
			&=\E(G\ind_{\tau<T}\Lambda[f](X^*_{\tau}))
			\end{align*}
		
		\end{proof}

	\subsection{Duality after explosion}\label{sec-explo}

	So far, we have focused on minimal processes, without considering what may happen after explosion. This is not enough in general for the $\Lambda$-linked process constructed previously to be a strong stationary dual for $X$.

	Indeed, let us consider the simple example where $X$ is a positive recurrent birth-death process on $\Z$. Proceeding like Diaconis and Fill \cite{Diaco}, we want to construct a dual $X^*$ whose states are the intervals of $\Z$, and linked to $X$ by a kernel $\Lambda$ which is the stationary distribution restricted to intervals. In that case, the state $\left[-\infty,+\infty\right]$ cannot be reached by $X^*$ in finitely many jumps. Since the time when this occurs is the wanted strong stationary time for $X$, we would like it to be finite, so we need to investigate what happens after explosion. This example will be studied in more details in the next section.

	\paragraph{}We are now going to specify a little more our framework. We consider a countable topological space $E$ which, for convenience, we will assume to be separated.

		\begin{defi}\label{def-strat}
		We say that a generator $\hat{\L}$ on $E$ is \emph{stratified} if there exists $N\in \N$ and a $N$-tuple $(E_i)_{i\in\dbc{1,N}}$ of subsets of $E$, called \emph{stratification} of $\hat{\L}$, such that:
  	
			\begin{itemize}
			\item The $E_i$ are mutually disjoint.
			\item $\forall i\in\dbc{1,N},\ \forall x\in E_i$, if $(Y_t^x)_{0\leq t\leq\tau}$ is a minimal process starting at $x$ with generator $\hat{\L}$, and explosion time $\tau$, then:
				$$\P(Y^x_t\in\bigsqcup_{j=1}^i E_j\ |\;t<\tau)=1$$
			\item If in addition $\tau<+\infty$, then almost surely there exists $j\leq i$ and $\varepsilon$ such that $Y_t^x\in E_j,\ \forall t\in\left]\tau-\varepsilon,\tau\right[$, and $\lim\limits_{t\rightarrow\tau}Y_t^x$ exists and is in $\bigsqcup_{k=j+1}^N E_k$. This limit will be denoted by $Y_{\tau}^x$.
			\end{itemize}
		\end{defi}
  	
  	This definition may be understood as follows: a minimal process whose generator is stratified can jump from one stratum to another only if the rank of the arrival stratum is lower than or equal to that of the departing one. On the other hand, at the explosion time the process must have a limit in a stratum with rank (strictly) greater than that of the stratum where it was right before. Naturally, we say that the process at the explosion time is equal to this limit, and starts again from there. That is done in the following construction.

		\begin{rem}
		To satisfy the second condition, we must have $\hat{\L}_{y,z}=0$ as soon as $y\in E_j,\; z\in E_k,\; k>j$. The $Q$-matrix $(\hat{\L}_{x,y})_{(x,y)\in E^2}$ is then a lower triangular block matrix, each block corresponding to a different stratum.
		\end{rem}

		Let $\hat{\L}$ be a stratified generator on $E$, $(E_i)_{i\in\dbc{1,N}}$ a stratification of $\hat{\L}$, with $N\in\N$, and $\nu$ a probability measure on $E$. We construct a non-minimal process $Z$ with generator $\hat{\L}$ and initial distribution $\nu$ as follows:
		
			\begin{itemize}
					
			\item We set $T_0=0$ and $Z_{T_0}=x$ with probability $\nu(x)$.
					
			\item By induction on $i\geq0$, we assume $Z$ to be constructed up to the time $T_i$. If $T_i=\infty$, the construction is over. Otherwise, we simulate a minimal process $Y^{(i)}$ with generator $\hat{\L}$ and starting from $Z_{T_i}$, independently from $Z_{[0,T_i]}$ given $Z_{T_i}$. Let $T'_{i+1}$ be its explosion time, we set:
				$$Z_{T_i+t}=Y_t^{(i)},\quad \forall t<T'_{i+1}$$
				$$T_{i+1}=T_i+T'_{i+1}$$
				$$Z_{T_{i+1}}=Y_{T'_{i+1}}^{(i)}\quad\text{if }T_{i+1}<\infty$$
					
			\end{itemize}

		\begin{defi}
		We call \emph{startified jump process} (resp. \emph{stratified semi-group}) \emph{associated to $(\nu,\hat{\L})$} any process $Z$ constructed as above (resp. its transition function). This process is defined up to a \emph{double explosion time}, possibly finite:
			$$T:=\lim\limits_{i\rightarrow+\infty}T_i$$

		As for minimal processes, we extend it after this double explosion time by setting $X_t=\Delta$ for every $t\geq T$, where $\Delta$ is a cemetery point not in $E$.
		
		\end{defi}
		By construction, such a process satisfies the strong Markov property, and by Anderson \cite{anderson} Proposition 1.2.7, its transition function is solution of the backward equation for $\hat{\L}$ (but may not be solution of the forward equation).

		\begin{rem}
		We could also define the process after the double explosion time by starting again from the limit, up to a triple explosion time, and so on through multiple explosion times, provided that the limits exist. We will see that for our purpose this is not necessary.
      
		At each explosion time, the process goes from one stratum $E_i$ toward another one $E_j$, with $j>i$, and such a transition is only possible after an explosion.
      
		\end{rem}

	The next proposition fully justifies the construction we have just made:
		
		\begin{prop}\label{martingale}
		Let $\hat{\L}$ be a stratified generator on $E$ and $Z$ a stratified jump process relative to $\hat{\L}$, with explosion times $T_n,\;n\in\N$. For every bounded, continuous function $f:E\rightarrow\R$, define:
			$$M^f_t=f(Z_{t})-f(Z_0)-\int_{0}^{t}\hat{\L}f(Z_s)\mathrm{d}s$$
		If $\hat{\L}f$ is bounded then for all $n\in\N,\;(M^f_{t\wedge T_n})_{t>0}$ is a martingale with respect to the filtration $(\F^{Z}_t)_{t>0}$. In particular, if the double-explosion time is infinite almost surely, $M^f$ is a martingale.
		\end{prop}
		
		\begin{proof}
		We prove by induction on $i$ that the processes $(M^f_{t\wedge T_i})_{t\geq0}$ are martingales. For $i=0$, this is trivial $(T_0=0)$. Let us assume then that this is true for some $i\geq0$. We set:
			$$Y_t=\begin{cases}
			Z_{T_i+t}\quad\text{if }T_i+t<T_{i+1}\\
			\Delta\quad\text{otherwise}
			\end{cases}$$
		Then $Y$ is a minimal process on $E$ with generator $\hat{\L}$. Call $\tau_n$ its $n$-th jumping time, $n\in\N$, we get for all $0<s<t$:
			
			\begin{align*}
			\E(M^f_{T_i+t\wedge\tau_n}-M^f_{T_i+s\wedge\tau_n}\ |\ \F^Z_{T_i+s\wedge\tau_n})&=\E(M^f_{T_i+t\wedge\tau_n}-M^f_{T_i+s\wedge\tau_n}\ |\ Z_{T_i+s\wedge\tau_n})\\
			&=\E(f(Y_{t\wedge\tau_n})-f(Y_{s\wedge\tau_n})-\int_{s\wedge\tau_n}^{t\wedge\tau_n}\hat{\L}f(Y_r)\mathrm{d}r\ |\ Y_{s\wedge\tau_n})\\
			&=0
			\end{align*}
		
		\noindent since $Y$ is a minimal process with generator $\hat{\L}$.
		
		Making $n$ go to infinity, we get by dominated convergence theorem that the process $(M^f_{(T_i+t)\wedge T_{i+1}})_{t\geq0}$ is a martingale. Thus:
			
			\begin{align*}
			\E(M^f_{t\wedge T_{i+1}}-M^f_{s\wedge T_{i+1}}\ |\ \F_s)&=\E(M^f_{t\wedge T_{i+1}}-M^f_{s\wedge T_{i+1}}\ |\ \F_s)\;\ind_{s<T_i}\\
			&=\E(\E(M^f_{t\wedge T_{i+1}}-M^f_{s\wedge T_{i+1}}\ |\ \F_{T_i})\ |\ \F_s)\;\ind_{s<T_i}\\
			&=\E(\E(M^f_{t\wedge T_{i+1}}\ |\ \F_{T_i})-M^f_{s\wedge T_{i+1}}\ |\ \F_s)\;\ind_{s<T_i}\\
			&=\E(M^f_{T_i}\ind_{t>T_i}+M^f_{t}\ind_{t<T_i}-M^f_{s\wedge T_{i+1}}\ |\ \F_s)\;\ind_{s<T_i}\\
			&=\E(M^f_{t\wedge T_i}-M^f_{s\wedge T_i}\ |\ \F_s)\;\ind_{s<T_i}\\
			&=0
			\end{align*}
		
		\noindent If $T_i\underset{i\rightarrow \infty}{\rightarrow}\infty$, we use dominated convergence once again to get that $(M^f_{t})_{t\geq0}$ is a martingale.
		\end{proof}

	\paragraph{}From now on we assume that:

		\begin{itemize}
		
		\item $\S$ is a discrete topological space.
		
		\item There exists a stratification $(\S^*_i)_{i\in\dbc{1,N}},\ N\in\N$ of $\L^*$ on $\S^*$. We write $P^*(\cdot)$ for the stratified semi-group associated to $\L^*$. 
		
		\item The double-explosion time of any stratified jump process associated to $\L^*$ is infinite almost surely. The matrix $(P^*_{x^*,y^*}(t))_{x^*,y^*\in\S^*}$ is then stochastic for all $t>0$.
				
		\item For all $x\in\S,\;\S^*\ni x^*\rightarrow\Lambda(x^*,x)$ is continuous.
		
		\end{itemize}

		\begin{Lem}
		Under these hypotheses, $(\S\times \S^*_i)_{i\in\dbc{1,N}}$ is a stratification of $\bar{\L}$ on $\S\times\S^*$ endowed with the product topology.
		\end{Lem}

		\begin{proof}
	
		This follows trivially from the fact that $\L$ is non-explosive and that the marginals of a minimal process $(X,X^*)$ with generator $\bar{\L}$ have generators $\L$ and $\L^*$ respectively.
  		
		\end{proof}

	\paragraph{}The next theorem will be an indispensable tool to construct strong stationary duals on countable spaces.

		\begin{thm}\label{cond-dual2}
		Under the above hypotheses, if:

			$$\begin{cases}
			\mu_0^*\Lambda=\mu_0\\
			\L^*\Lambda=\Lambda\L
			\end{cases}$$

		\noindent then there exists a stratified jump process $(X_t^*)_{t>0}$ on $\mathcal{S}^*$ with generator $\L^*$ and initial distribution $\mu_0^*$, such that for all $t\geq0$:

			\begin{gather}
			\l(X_t\ |\ \mathcal{F}^{X^*}_t)=\Lambda(X^*_t,\cdot)\label{indep-3}\\	
			\l(X^*_t\ |\ X)=\l(X^*_t\ |\ \mathcal{F}^X_t)\label{indep-4}
			\end{gather}

		\end{thm}

		\begin{proof}Let $X^*$ be a stratified jump process constructed from minimal processes given by Theorem \ref{cond-dual}. Thus $(X,X^*)$ has generator $\bar{\L}$ (defined in (\ref{gen-couple})) and initial distribution $\bar{\mu}$ (we recall that $\bar{\mu}(x_0,x_0^*)=\mu_0^*(x_0^*)\Lambda(x_0^*,x_0)$). By construction, $X^*$ satisfies (\ref{indep-4}). Let us show that it satisfies (\ref{indep-3}).
		
		We set $T_0=0$ and for each $i\geq0$ we call $T_{i+1}^k$ the $k$-th jumping time of $(X,X^*)$ following $T_i,\ k\in\N$ and:
			$$T_{i+1}:=\lim\limits_{k\rightarrow\infty}T_{i+1}^k$$
		the $(i+1)$-th explosion time. From hypotheses,
			$$\lim_{i\rightarrow+\infty}T_i=+\infty$$
		We use the convention $T_j^k=+\infty,\ \forall k,\ \forall j>i$ if $T_i=+\infty$.
		
		First, we prove that the process $(X,X^*)$ is almost surely continuous at each explosion time $T_i$. Indeed, $X^*$ is continuous at $T_i$ by definition of a stratified process, and $X$ is discontinuous only when it jumps. Let us call $\sigma_n$ the jumping times of $X$, assume by induction on $n$ that $(X,X^*)$ is continuous at each explosion time until $\sigma_n$, and call $T_i$ the last explosion before $\sigma_n$. Let $\tilde{\L}$ be a generator on $\S^*$ defined by $\tilde{\L}_{x^*,y^*}=\bar{\L}_{(X_{\sigma_n},x^*),(X_{\sigma_n},y^*)}$ and $Y$ be a stratified process starting from $X^*_{\sigma_n}$ and with generator $\tilde{\L}$. Looking at the construction of $X^*$ we made in the proof of Theorem \ref{cond-dual}, we see that $(X^*_{\sigma_n+t})_t$ is equal in law to $Y$ until the time $\sigma_{n+1}$. Calling $\tau_k$ the jumping times of $Y$ before its first explosion, we get:
		
			\begin{align*}
			\P(T_{i+1}=\sigma_{n+1}\,|\,\F^{\bar{X}}_{\sigma_n})&=\P(\sigma_n+\tau_k\overset{k\rightarrow\infty}{\longrightarrow}\sigma_{n+1}\,|\,\F^{\bar{X}}_{\sigma_n})\\
			&=\P(\sigma_n+\tau_1+\sum\limits_{k=1}^{\infty}(\tau_{k+1}-\tau_k)=\sigma_{n+1}\,|\,\F^{\bar{X}}_{\sigma_n})\\
			&=0
			\end{align*}
			
		\noindent since $\tau_1$ is a continuous random variable, independent from the $\tau_{k+1}-\tau_k,\,\sigma_n$ and $\sigma_{n+1}$ conditionally to $\F^{\bar{X}}_{\sigma_n}$. Proceeding the same way (by a second level of induction), starting from $T_{i+j},\;j\ge0$, if it is smaller than $\sigma_{n+1}$, we conclude the induction. In conclusion, we get $\sigma_n\ne T_m,\,\forall m,n\ge0$.
		
		\paragraph{}We will now show by induction on $i$ that for all $0=t_1<\dots<t_n<\infty$ and all $(x_1^*,\dots,x_n^*,x_n)\in\{\mathcal{S}^*\}^n\times\S,\ n\in\N$, we have:
		
			\begin{align}
			&\P(X^*_{t_1}=x^*_1,\dots,X^*_{T_i}=x^*_n,X_{T_i}=x_n\ |\ t_{n-1}<T_i\leq t_n)\nonumber\\
			=&\P(X^*_{t_1}=x^*_1,\dots,X^*_{T_i}=x^*_n\ |\ t_{n-1}<T_i\leq t_n)\Lambda(x^*_n,x_n)\label{dual-explo}
			\end{align}
			
		\noindent and:
		
			\begin{align}
			&\P(X^*_{t_1}=x^*_1,\dots,X^*_{t_n}=x^*_n,X_{t_n}=x_n\ |\ T_{i+1}>t_n)\nonumber\\
			=&\P(X^*_{t_1}=x^*_1,\dots,X^*_{t_n}=x^*_n\ |\ T_{i+1}>t_n)\Lambda(x^*_n,x_n)\label{dual-apres-explo}
			\end{align}

		\paragraph{}For $i=0$, we simply consider a minimal process with generator $\bar{\L}$ so the base case is given by the initial distribution of $(X,X^*)$ (for (\ref{dual-explo})) and Theorem \ref{cond-dual} (for (\ref{dual-apres-explo})). Suppose the induction hypothesis true for some $i-1\geq0$. From (\ref{dual-apres-explo}) we deduce:
			$$\l(X_t\ |\ \F_t^{X^*},T_i>t)=\Lambda(X^*_{t},\cdot),\quad\forall t>0$$
		and in the same way as in Corollary \ref{cor-tps-darret}:
			$$\l(X_S\ |\ \F_S^{X^*})=\Lambda(X^*_S,\cdot)$$
		for every stopping time $S<T_i$ almost surely.
		
		We set:
			$$\mathcal{A}_i:=\bigvee_{k\geq0}\F_{T_i^k}^{X^*}$$
		the $\sigma$-algebra generated by the union of $\F_{T_i^k}^{X^*}$.
		
		Let $G$ be a bounded, real-valued random variable, measurable with respect to $\mathcal{A}_i$, and $f\colon\S\rightarrow\R$ be a finitely supported function. Then $\Lambda[f]$ is continuous and bounded, and by the convergence of closed martingales theorem, we have $\E(G\ |\ \F_{T_i^k}^{X^*})\rightarrow G$ almost surely. Thus using the dominated convergence theorem:

			\begin{align*}
			\E(Gf(X_{T_i}))&=\lim\limits_{k\rightarrow\infty}\E(\E(G\ |\ \F_{T_i^k}^{X^*})\ f(X_{T_i^k}))\\
			&=\lim\limits_{k\rightarrow\infty}\E(\E(G\ |\ \F^{X^*}_{T_i^k})\Lambda[f](X^*_{T_i^k}))\\
			&=\E(G\Lambda[f](X^*_{T_i}))
			\end{align*}

		Now, if we choose $0=t_1<\dots<t_n<\infty$ and $(x_1^*,\dots,x_n^*)\in\{\mathcal{S}^*\}^n,\ n\in\N$ then there exists a decreasing sequence $(V_m)_{m\in\N}$ of neighborhoods of $x^*_n$ such that $\underset{m}{\bigcap} V_m=\{x^*_n\}$ and we have:
		
			\begin{align*}
			&\{X^*_{t_1}=x^*_1,\dots,X^*_{T_i}=x^*_n,T_i\in\left]t_{n-1},t_n\right]\}\\
			=&\{X^*_{t_1}=x^*_1,\dots,\lim\limits_{k\rightarrow\infty}X^*_{T^k_{i}}=x^*_n,\lim\limits_{k\rightarrow\infty}T^k_{i}\in \left]t_{n-1},t_n\right]\}\\
			=&\underset{m\geq0}{\bigcap}\;\underset{K\geq0}{\bigcup}\;\underset{k\geq K}{\bigcap}\{X^*_{t_1}=x^*_1,\dots,X^*_{T^k_{i}}\in V_m,T^k_{i}\in\left]t_{n-1},t_n\right]\}\\
			\end{align*}
		
		\noindent that is:
			$$\{X^*_{t_1}=x^*_1,\dots,X^*_{T_i}=x^*_n,T_i\in\left]t_{n-1},t_n\right]\}\in\mathcal{A}_{i}$$
		and we deduce (\ref{dual-explo}). In other words, the intertwining relation is still true "at explosion time $T_i$".
		
		By construction, the process $((X,X^*)_{T_i+t})_{t\geq0}$ has the same law as a minimal process $(Y,Y^*)$ with generator $\bar{\L}$ up to the explosion time $T_{i+1}$. This process is independent from $((X,X^*)_{T_i\wedge t})_{t\geq0}$ given $(X,X^*)_{T_i}$ and thanks to what we have done right before, we also know that the initial distribution of this process satisfies:
			$$\P(Y_0=x_0\ |\ Y^*_0=x^*_0)=\Lambda(x^*_0,x_0)$$
		for every $(x_0,x^*_0)\in\bar{\S}$. From Theorem \ref{cond-dual} we deduce that for all $0=t_1<\dots<t_n<\infty$ and all $(x_1^*,\dots,x_n^*,x_n)\in\{\mathcal{S}^*\}^n\times\S,\ n\in\N$:
			
			\begin{align*}
			&\P(X^*_{t_1}=x^*_1,\dots,X^*_{t_n}=x^*_n,X_{t_n}=x_n\ |\ T_{i+1}>t_n,\;T_i\leq t_1)\\
			=&\P(Y^*_{t_1-T_i}=x^*_1,\dots,Y^*_{t_n-T_i}=x^*_n,Y_{t_n-T_i}=x_n\ |\ T^Y>t_n-T_i,\;T_i\leq t_1)\\
			=&\P(Y^*_{t_1-T_i}=x^*_1,\dots,Y^*_{t_n-T_i}=x^*_n\ |\ T^Y>t_n-T_i,\;T_i\leq t_1)\Lambda(x^*_n,x_n)\\
			=&\P(X^*_{t_1}=x^*_1,\dots,X^*_{t_n}=x^*_n\ |\ T_{i+1}>t_n,\;T_i\leq t_1)\Lambda(x^*_n,x_n)
			\end{align*}
		
		\noindent where $T^Y$ is the first explosion time of the process $(Y,Y^*)$. We use it to show that the intertwining relation is still valid after the explosion time $T_i$. Indeed, if $k\in\dbc{1,n-1}$, then using:
			\begin{enumerate}[i.]
			\item the Markovianity of $(X,X^*)$\label{1}
			\item the independence of $X_{T_i}$ and $X^*_{[0,T_i]}$ given $X^*_{T_i}$\label{2}
			\item the previous computation\label{3}
			\end{enumerate}
		we get:		
			\begin{align*}
			&\P(X^*_{t_1}=x^*_1,\dots,X^*_{t_n}=x^*_n,X_{t_n}=x_n\ |\ T_i\in\left]t_k,t_{k+1}\right],\;T_{i+1}>t_n)\\
			=&\sum_{\bar{x}\in\bar{\S}}\P(X^*_{t_1}=x^*_1,\dots,X^*_{t_n}=x^*_n,X_{t_n}=x_n,\bar{X}_{T_i}=\bar{x}\ |\ T_i\in\left]t_k,t_{k+1}\right],\;T_{i+1}>t_n)\\
			\overset{\ref{1}.}{=}&\sum_{\bar{x}\in\bar{\S}}\P(X^*_{t_1}=x^*_1,\dots,X^*_{t_{k}}=x^*_{k},\bar{X}_{T_i}=\bar{x}\ |\ T_i\in\left]t_k,t_{k+1}\right],\;T_{i+1}>t_n)\\
			&\times\P(X^*_{t_{k+1}}=x^*_{k+1},\dots,X^*_{t_n}=x^*_n,X_{t_n}=x_n\ |\ \bar{X}_{T_i}=\bar{x},\;T_i\in\left]t_k,t_{k+1}\right],\;T_{i+1}>t_n)\\
			\overset{\ref{2}.}{=}&\sum_{\bar{x}\in\bar{\S}}\P(X^*_{t_1}=x^*_1,\dots,X^*_{t_{k}}=x^*_{k},X^*_{T_i}=x^*\ |\ T_i\in\left]t_k,t_{k+1}\right],\;T_{i+1}>t_n)\\
			&\quad\times\P(X_{T_i}=x\ |\ X^*_{T_i}=x^*,\;T_i\in\left]t_k,t_{k+1}\right],\;T_{i+1}>t_n)\\
			&\quad\times\dfrac{\P(X_{T_i}=x,X^*_{t_{k+1}}=x^*_{k+1},\dots,X_{t_n}=x_n\ |\ X^*_{T_i}=x^*,\;T_{i+1}>t_n,\;T_i\in\left]t_k,t_{k+1}\right])}{\P(X_{T_i}=x\ |\ X^*_{T_i}=x^*,\;T_i\in\left]t_k,t_{k+1}\right],\;T_{i+1}>t_n)}\\
			=&\sum_{\bar{x}\in\bar{\S}}\P(X^*_{t_1}=x^*_1,\dots,X^*_{t_{k}}=x^*_{k},X^*_{T_i}=x^*\ |\ T_i\in\left]t_k,t_{k+1}\right],\;T_{i+1}>t_n)\\
			&\quad\times\P(X_{T_i}=x,X^*_{t_{k+1}}=x^*_{k+1},\dots,X_{t_n}=x_n\ |\ X^*_{T_i}=x^*,\;T_i\in\left]t_k,t_{k+1}\right],\;T_{i+1}>t_n)\\
			=&\sum_{x^*\in{\S^*}}\;\sum_{x\in{\S}}\P(X^*_{t_1}=x^*_1,\dots,X^*_{t_{k}}=x^*_{k},X^*_{T_i}=x^*\ |\ T_i\in\left]t_k,t_{k+1}\right],\;T_{i+1}>t_n)\\
			&\quad\times\P(X_{T_i}=x,X^*_{t_{k+1}}=x^*_{k+1},\dots,X_{t_n}=x_n\ |\ X^*_{T_i}=x^*,\;T_i\in\left]t_k,t_{k+1}\right],\;T_{i+1}>t_n)\\
			=&\sum_{x^*\in\S^*}\P(X^*_{t_1}=x^*_1,\dots,X^*_{t_{k}}=x^*_{k},X^*_{T_i}=x^*\ |\ T_i\in\left]t_k,t_{k+1}\right],\;T_{i+1}>t_n)\\
			&\quad\times\P(X^*_{t_{k+1}}=x^*_{k+1},\dots,X_{t_n}=x_n\ |\ X^*_{T_i}=x^*,\;T_i\in\left]t_k,t_{k+1}\right],\;T_{i+1}>t_n)\\
%
%
			\overset{\ref{3}.}{=}&\sum_{x^*\in\S^*}\P(X^*_{t_1}=x^*_1,\dots,X^*_{t_{k}}=x^*_{k},X^*_{T_i}=x^*\ |\ T_i\in\left]t_k,t_{k+1}\right],\;T_{i+1}>t_n)\\
			&\quad\times\P(X^*_{t_{k+1}}=x^*_{k+1},\dots,X^*_{t_n}=x^*_n\ |\ X^*_{T_i}=x^*,\;T_i\in\left]t_k,t_{k+1}\right],\;T_{i+1}>t_n)\;\Lambda(x^*_n,x_n)\\
			\end{align*}
		Keeping $\Lambda(x^*_n,x_n)$ apart, we do the same steps backward to get that this quantity is equal to:
			$$\P(X^*_{t_1}=x^*_1,\dots,X^*_{t_n}=x^*_n\ |\ T_i\in\left]t_k,t_{k+1}\right],\;T_{i+1}>t_n)\;\Lambda(x^*_n,x_n)$$
		so the only thing left is to sum on all possible intervals for $T_i$:
		
			\begin{align*}
			&\P(X^*_{t_1}=x^*_1,\dots,X^*_{t_n}=x^*_n,X_{t_n}=x_n\ |\ T_{i+1}>t_n)\\
			=&\P(X^*_{t_1}=x^*_1,\dots,X^*_{t_n}=x^*_n,X_{t_n}=x_n\ |\ T_{i+1}>t_n,\;T_i>t_{n})\;\P(T_i>t_{n}\ |\ T_{i+1}>t_n)\\
			&\quad+\sum_{k=1}^{n-1}\P(X^*_{t_1}=x^*_1,\dots,X^*_{t_n}=x^*_n,X_{t_n}=x_n,\;t_{k}<T_i\leq t_{k+1}\ |\ T_{i+1}>t_n)\\
			=&\P(X^*_{t_1}=x^*_1,\dots,X^*_{t_n}=x^*_n\ |\ T_{i}>t_{n})\;\Lambda(x^*_n,x_n)\;\P(T_i>t_{n}\ |\ T_{i+1}>t_n)\\
			&\quad+\sum_{k=1}^{n-1}\P(X^*_{t_1}=x^*_1,\dots,X^*_{t_n}=x^*_n,\;t_{k}<T_i\leq t_{k+1}\ |\ T_{i+1}>t_n)\;\Lambda(x^*_n,x_n)\\
			=&\P(X^*_{t_1}=x^*_1,\dots,X^*_{t_n}=x^*_n\ |\ T_{i+1}>t_n)\;\Lambda(x^*_n,x_n)\\
			\end{align*}
		
		\noindent which concludes the induction. We deduce:
			$$\l(X_t\ |\ \F_t^{X^*},\;T_i>t)=\Lambda(X^*_t,\cdot)$$
		for every $i\in\N$, that is:
			$$\E(\ind_{T_i>t}\;G\;f(X_t))=\E(\ind_{T_i>t}\;G\;\Lambda\left[f\right](X^*_t))$$
		for every bounded continuous function $f$ and every bounded random variable $G$, measurable with respect to $\F_t^{X^*}$. Now, $\underset{i\rightarrow\infty}{\lim}\ind_{T_i>t}=1$ almost surely and we conclude by dominated convergence:
			$$\l(X_t\ |\ \F_t^{X^*})=\Lambda(X^*_t,\cdot)$$
		\end{proof}
		
		\begin{cor}\label{cor-tps-absorption}
		The relation (\ref{indep-3}) may be replaced by:
			
			\begin{equation}
			\l(X_T\ |\ \F^{X^*}_T,\;T<\infty)=\Lambda(X^*_T,\cdot)\label{indep-5}
			\end{equation}
		
		\noindent for every $\F^{X^*}_t$-stopping time $T$.
		\end{cor}
		
		\begin{proof}
		Condition (\ref{indep-5}) trivially implies (\ref{indep-3}). The converse is the same as for Corollary \ref{cor-tps-darret}.
		\end{proof}

\section{Application to birth-death processes}\label{sec-marche-Z}
Let $(X_t)_{t\geq0}$ be a birth-death process on $\Z$, with death rates $(a_n)_{n\in\Z}$ and birth rates $(b_n)_{n\in\Z}$, all positive. We assume that the coefficients are such that $X$ is nonexplosive and positive recurrent, that is (cf Anderson \cite{anderson} Chapter 8):

	\begin{gather}
	\sum_{i=-\infty}^{+\infty}\mu(i)<\infty\\
	\begin{cases}
	\sum_{i=0}^{+\infty}\dfrac{1}{b_i\mu(i)}=\infty\\
	\sum_{i=-\infty}^{0}\dfrac{1}{b_i\mu(i)}=\infty
	\end{cases}
	\end{gather}

\noindent where $\mu$ is the stationary distribution, uniquely defined by the relations:

	\begin{gather}
	\mu(n)b_n=\mu(n+1)a_{n+1}\\
	\sum_{n\in\Z}\mu(n)=1
	\end{gather}
	
\noindent namely, for every $n\in\N^*$:

	\begin{equation*}
	\mu(n)=\mu(0)\prod_{i=0}^{n-1}\dfrac{b_i}{a_{i+1}},\quad
	\mu(-n)=\mu(0)\prod_{i=0}^{n-1}\dfrac{a_{-i}}{b_{-i-1}}
	\end{equation*}
	
\noindent with $\mu(0)=\dfrac{1}{1+\sum_{n\in\N^*}\left(\prod_{i=0}^{n-1}\dfrac{b_i}{a_{i+1}}+\prod_{i=0}^{n-1}\dfrac{a_{-i}}{b_{-i-1}}\right)}$.

We assume that the coefficients are such that:
	$$\sum_{n\in\Z}\mu(n)b_n<\infty$$

The generator $(\L_{i,j})_{(i,j)\in\Z^2}$ of this process $X$ is defined by:

	\begin{equation}
	\forall i,j\in\Z,\quad\L_{i,j}=
		\begin{cases}
		b_i&\text{ if }j=i+1\\
		a_i&\text{ if }j=i-1\\
		-a_i-b_i&\text{ if }j=i\\
		0&\text{ otherwise}
		\end{cases}
	\end{equation}

We denote $\mu_t$ the distribution of $X_t,\; \forall t\geq0$.

\paragraph{}Starting with this process, we want to construct a strong stationary dual of $X$, with values in:
	$$\mathbf{E}:=\{(p,q)\in(\Z\cup\{-\infty\})\times(\Z\cup\{+\infty\}):\;p\leq q\}$$
the set of all interval of $\Z$, linked to $X$ by the kernel $\Lambda$ defined by:
	$$\Lambda((p,q),n)=\delta_n(\dbc{p,q})\dfrac{\mu(n)}{\mu(\dbc{p,q})},\quad \forall(p,q)\in\mathbf{E},\quad \forall n\in\Z$$
Here and in the sequel, for $(p,q)\in\mathbf{E}$, $n\in\dbc{p,q}$ means $n\in[p,q]\cap\Z$ (we exclude $n=\pm\infty$). 

We endow $\mathbf{E}$ with the topology generated by the following open sets:
	$$\{(p,q)\}$$
	$$\{(p,q+i),\;i\in\N\}\cup\{(p,+\infty)\}$$
	$$\{(p-i,q),\;i\in\N\}\cup\{(-\infty,q)\}$$
	$$\{(-\infty,q+i),\;i\in\N\}\cup\{\Z\}$$
	$$\{(p-i,+\infty),\;i\in\N\}\cup\{\Z\}$$
	$$\{(p-i,q+j),\;i,j\in\N\}\cup\{\Z\}$$
with $-\infty<p\leq q<+\infty$. We set:
	$$\mathbf{E}_1:=\{(p,q)\in\Z^2:\;p\leq q\}$$
	$$\mathbf{E}_2:=\{(-\infty,q)\in\mathbf{E}\}\cup\{(p,+\infty)\in\mathbf{E}\}$$
	$$\mathbf{E}_3:=\{\Z\}$$

Let $\L^*$ be a generator on $\mathbf{E}$ defined as follows:

	\begin{equation}
	\L^*((p,q),(p_i,q_i))=
		\begin{cases}
		\dfrac{\mu(\dbc{p-1,q})}{\mu(\dbc{p,q})}b_{p-1}=:a_{p,q}' & \text{ if }(p_i,q_i)=(p-1,q)\smallskip\\
		\dfrac{\mu(\dbc{p,q+1})}{\mu(\dbc{p,q})}a_{q+1}=:b_{p,q} & \text{ if }(p_i,q_i)=(p,q+1)\smallskip\\
		\dfrac{\mu(\dbc{p+1,q})}{\mu(\dbc{p,q})}a_p=:b_{p,q}' & \text{ if }(p_i,q_i)=(p+1,q)\smallskip\\
		\dfrac{\mu(\dbc{p,q-1})}{\mu(\dbc{p,q})}b_q=:a_{b,q} & \text{ if }(p_i,q_i)=(p,q-1)\smallskip\\
		0 \quad &\text{ for all other }(p_i,q_i)\neq(p,q)
		\end{cases}
	\end{equation}

for all $(p,q)\in\mathbf{E}$, with the convention $a_\infty=b_\infty=a_{-\infty}=b_{-\infty}=0$. Then we get:

	\begin{prop}
	There exists a measure $\mu_0^*$ on $\mathbf{E}$ and a minimal process $X^*$ with initial distribution $\mu_0^*$ and generator $\L^*$ such that $X^*$ is $\Lambda$-linked to $X$.
	\end{prop}

	\begin{proof}
	From Theorem \ref{cond-dual}, it is enough to have $\mu_0^*\Lambda=\mu_0$ and $\L^*\Lambda=\Lambda\L$. For the first one it is immediate if we choose the measure whose support is the set of singletons: $\mu_0^*((m,n))=\delta_{m,n}\;\mu_0(n)$. For the second one: $\forall (p, q)\in\mathbf{E},\; \forall n\in\Z$

		\begin{align*}
		(\Lambda\L)_{(p,q),n}&=\sum_{m\in\Z}\Lambda((p,q),m)\;\L(m,n)\\
		&=\sum_{m\in\dbc{p,q}}\Lambda((p,q),m)\;\L(m,n)\\
		&=\dfrac{\mu(n-1)}{\mu(\dbc{p,q})}\;\delta_{n-1}(\dbc{p,q})\;b_{n-1}\\&\quad\quad+\dfrac{\mu(n+1)}{\mu(\dbc{p,q})}\;\delta_{n+1}(\dbc{p,q})\;a_{n+1}-\dfrac{\mu(n)}{\mu(\dbc{p,q})}\;\delta_{n}(\dbc{p,q})\;(b_n+a_n)\\
		&=\dfrac{\mu(n)[b_n(\delta_{n,p-1}-\delta_{n,q})+a_n(\delta_{n,q+1}-\delta_{n,p})]}{\mu(\dbc{p,q})}
		\end{align*}

	\noindent and: 

		\begin{align*}
		(\L^*\Lambda)_{(p,q),n}&=\sum_{(p',q')\in\mathbf{E}}\L^*((p,q),(p',q'))\;\Lambda((p',q'),n)\\
		&=\dfrac{\mu(n)}{\mu(\dbc{p,q})}(b_{p-1}\delta_n(\dbc{p-1,q})-\delta_n(\dbc{p,q})b_{p-1}\\
		&\hspace{47pt}+a_{q+1}\delta_n(\dbc{p,q+1})-\delta_n(\dbc{p,q})a_{q+1}\\
		&\hspace{47pt}+a_p\delta_n(\dbc{p+1,q})-\delta_n(\dbc{p,q})a_p\\
		&\hspace{47pt}+b_q\delta_n(\dbc{p,q-1})-\delta_n(\dbc{p,q})b_q)\\
		&=\dfrac{\mu(n)[b_n(\delta_{n,p-1}-\delta_{n,q})+a_n(\delta_{n,q+1}-\delta_{n,p})]}{\mu(\dbc{p,q})}\\
		&=(\Lambda\L)_{(p,q),n}
		\end{align*}

	\end{proof}

So we have a minimal process $(X^*_t)=(P_t,Q_t)\in\mathbf{E}$. Its bounds $P$ and $Q$ evolve as birth-death processes individually not Markovian on $\Z$, until a possible finite explosion time $T$. Here, the only absorbing state for $X^*$ is $\Z$, but it is not reachable in finitely many jumps (if $X^*$ starts from another state). We hence have to consider a stratified jump process associated to $\L^*$, which is made possible by:

	\begin{prop}
	$\L^*$ is stratified on $\mathbf{E}$, and $(\mathbf{E}_1,\mathbf{E}_2,\mathbf{E}_3)$ is a stratification of $\L^*$.
	\end{prop}

	\begin{proof}
	For a minimal process $(X^*_t)_t=(P_t,Q_t)_t$ with generator $\L^*$ and initial distribution $\mu^*_0$, we call $T$ its first explosion time and we set:
		$$T_1:=\begin{cases}
		T&\quad\text{if }T\text{ is an explosion time for }P\\
		+\infty&\quad\text{otherwise}
		\end{cases}$$
		$$T_2:=\begin{cases}
		T&\quad\text{if }T\text{ is an explosion time for }Q\\
		+\infty&\quad\text{otherwise}
		\end{cases}$$
	so that $T=T_1\wedge T_2$.
		
	We will show that $\underset{t\rightarrow T}{\lim}\;Q_t=+\infty$ a.s. if $T=T_2$, and we deduce by symmetry that $\underset{t\rightarrow T}{\lim}\;P_t=-\infty$ if $T=T_1$. To do so, we consider the process $X^*=(P,Q)$ coupled with a (homogeneous) birth-death process $Z$ with same initial distribution as $Q$, corresponding to the evolution of the right bound of $X^*$ when the left bound is in $-\infty$. Such a coupling is given by:

		\begin{Lem}\label{couplage}
		If $T=T_2$, there exists a coupling $(X^*,Z)$ such that $Q_t\geq Z_t,\; \forall 0\leq t\leq T$.
		\end{Lem}

		\begin{proof}
		We are going to construct $(Z_t)_{t\geq0}$, such that $(-\infty,Z)$ is a Markov process with generator $\L^*$. We simulate $X^*_0=(P_0,Q_0)$ with its initial distribution $\mu_0^*$, and we set $Z_0=Q_0$. The process $Z$ is a birth-death process with death and birth rates respectively:
			$$\tilde{b}_n=\dfrac{\mu(\dbc{-\infty,n+1})}{\mu(\dbc{-\infty,n})}a_{n+1} \ \text{ and }\ \tilde{a}_n=\dfrac{\mu(\dbc{-\infty,n-1})}{\mu(\dbc{-\infty,n})}b_n$$
		An elementary computation shows that for all $(p,q)\in\mathbf{E},\;b_{p,q}\geq\tilde{b}_q$ and $a_{p,q}\leq\tilde{a}_q$
%
		that is, the births are on average quicker for the process $(Q_t)_t$, and the deaths are slower. To make the coupling, we proceed by letting the two processes evolve separately and independently as long as $Q_t>Z_t$. Then, when they join, we "force" $Z_t$ to go down if $Q_t$ does, and $Q_t$ to go up if $Z_t$ does.

		For this purpose, we will define by iteration the sequence $(T_i)_i$ of junction times of $Q_t$ and $Z_t$, and $(T_i')_i$ that of the jumping times following the junction (if $Q_{T'_i}=Z_{T'_i}$ we set $T_{i+1}=T_i'$), and we will describe the behavior of the coupling $(X^*,Z)$ between those times. Let $T_0=0$. Assume that $(X^*,Z)$ is constructed until the time $T_i$, for some $i\geq 0$. We define $T_i'$ as a random variable such that $T_i'-T_i$ is exponentially distributed with parameter $\lambda_i:=\tilde{a}_{Q_{T_i}}+b_{P_{T_i},Q_{T_i}}+a_{P_{T_i},Q_{T_i}}'+b_{P_{T_i},Q_{T_i}}'$.
		We impose $Q_t=Z_t=Q_{T_i}$ and $P_t=P_{T_i}, \ \forall T_i<t<T_i'$, then: 
			$$((P_{T_i'},Q_{T_i'}),Z_{T_i'})=
				\begin{cases}
				((P_{T_i},Q_{T_i}-1),Z_{T_i}-1) & \text{ with probability } \dfrac{a_{P_{T_i},Q_{T_i}}}{\lambda_i}\\
				((P_{T_i},Q_{T_i}+1),Z_{T_i}+1) & \text{ with probability } \dfrac{\tilde{b}_{Q_{T_i}}}{\lambda_i}\smallskip\\
				((P_{T_i},Q_{T_i}),Z_{T_i}-1) & \text{ with probability } 	\dfrac{\tilde{a}_{Q_{T_i}}-a_{P_{T_i},Q_{T_i}}}{\lambda_i}\\
				((P_{T_i},Q_{T_i}+1),Z_{T_i}) & \text{ with probability } \dfrac{b_{P_{T_i},Q_{T_i}}-\tilde{b}_{Q_{T_i}}}{\lambda_i}\\
				((P_{T_i}-1,Q_{T_i}),Z_{T_i}) & \text{ with probability } \dfrac{a_{P_{T_i},Q_{T_i}}'}{\lambda_i}\\
				((P_{T_i}+1,Q_{T_i}),Z_{T_i}) & \text{ with probability } \dfrac{b_{P_{T_i},Q_{T_i}}'}{\lambda_i}\\
				\end{cases}$$
		Next, $X^*_t$ and $Z_t$ evolve independently between times $T_i'$ and $T_{i+1}:=\inf\{t\geq T_i': Q_t=Z_t\}$. We clearly have $Q_t\geq Z_t,\; \forall 0\leq t\leq T$. Moreover, one checks that the marginals are the ones we want: write the coupling generator from the previous coefficients, and apply it to functions depending only on one variable.
%
%
		\end{proof}

	\paragraph{}It only remains to check: 
		
		\begin{equation}
		Z_t\underset{t\rightarrow+\infty}{\rightarrow}+\infty\quad\text{a.s.}\label{vie-mort-Z}
		\end{equation}
	
	$(-\infty,Z)$ is equal in law to a process $\tilde{X}^*$, $\Lambda$-linked to some $\tilde{X}$ with generator $\L$. The latter is positive recurrent, and for all $t\geq0, \tilde{X}_t\in\tilde{X}^*_t$ a.s., so $Z_t>0$ infinitely often (considering the underlying jump chain). This implies in particular that $Z_t\underset{t\rightarrow+\infty}{\rightarrow}-\infty$ is impossible a.s..

	If we consider the same process $\tilde{Z}_t$, forced to stay on $\N$ by setting $a_0=0$, then the condition for $\tilde{Z}_t\underset{t\rightarrow+\infty}{\rightarrow}+\infty$ is:
		$$\sum_{i=0}^{\infty}\prod_{j=1}^{i}\dfrac{\tilde{a}_j}{\tilde{b}_j}<\infty$$
	Now:

		\begin{align*}
		\sum_{i=0}^{\infty}\prod_{j=1}^{i}\dfrac{\tilde{a}_j}{\tilde{b}_j}&=\sum_{i=0}^{\infty}\prod_{j=1}^{i}\dfrac{\mu(\dbc{-\infty,j-1})b_j}{\mu(\dbc{-\infty,j+1})a_{j+1}}\\
		&\leq\sum_{i=0}^{\infty}\prod_{j=1}^{i}\dfrac{b_j}{a_{j+1}}
		<\infty
		\end{align*}

	Consequently, $\tilde{Z}_t\underset{t\rightarrow+\infty}{\rightarrow}\infty$ almost surely, and so does $Z_t$. Indeed, as long as $Z_t>0$, the process $Z$ behaves like $\tilde{Z}$. If we set $\tau_0=0,\;\sigma_0=\inf\{t\geq\tau_0,\;Z_t=0\}$, and for every $i>0$:
		$$\tau_i=\inf\{t>\sigma_{i-1},\;Z_t=1\}\quad\text{and}\quad\sigma_i=\inf\{t>\tau_i,\;Z_t=0\}$$
	(with the convention $\inf\varnothing=+\infty$), we have $\P(\tau_i<\infty\ |\ \sigma_{i-1}<\infty)=1$ since $Z_t>0$ infinitely often and:
		$$\P(\sigma_i=\infty\ |\ \tau_i<\infty)=\P(\tilde{Z}_t\neq0,\;\forall t>0\ |\ \tilde{Z}_0=1)>0$$
	So there exists almost surely a time $\tau_i$ for which $Z_t>0$ for all $t\geq\tau_i$ (succession of independent Bernoulli variables). Since:
		$$\l(Z_{\tau_i+\cdot}\ |\ \tau_i<\infty,\sigma_i=\infty)=\l(\tilde{Z}\ |\ \tilde{Z}_0=1,\tilde{Z}_t>0,\forall t>0)$$
	it follows that $Z_t\underset{t\rightarrow+\infty}{\rightarrow}\infty$ almost surely.
	
	\paragraph{}If on the opposite $T<T_2$, then $Q$ does not explode at time $T$ and thus admits a last jumping time $T_k$ before $T$. The same holds for $P$, and we conclude that $\underset{t\rightarrow T}{\lim}\;X^*_t$ exists almost surely.
	
	One easily checks that all three conditions for $\mathbf{E}_1,\mathbf{E}_2,\mathbf{E}_3$ to be a stratification of $\L^*$ are satisfied.
	\end{proof}

\paragraph{}
By Theorem \ref{cond-dual2}, we can now consider a stratified jump process $X^*$ with generator $\L^*$ and initial distribution $\mu_0^*$ satisfying $\mu^*_0\Lambda=\mu_0$, such that $X^*$ is $\Lambda$-linked to $X$. The interest of this process is that the existence of a finite strong stationary time for an initial distribution $\mu_0$ of $X$ is equivalent to the absorption in finite time of $X^*$ in $\Z$, for at least one $\mu_0^*$. The first implication derives directly from the definition of $\Lambda$ and Corollary \ref{cor-tps-absorption}, the converse comes from the following proposition:

	\begin{prop}\label{sharp}
	If $\mu_0=\mu_{|\dbc{-\infty,0}}$ and $X_0^*=(-\infty,0)$ a.s., then the dual $X^*$ is \emph{sharp}.
	\end{prop}

For the proof, refer to Miclo \cite{Miclo}, Lemma 26.

%
%

So, in that case, the time to absorption in $\Z$ of $X^*$ is a time to stationarity and its almost sure finiteness is equivalent to the existence of a finite strong stationary time (at least for that initial distribution $\mu_0$). 

We are finally able to state the main result of this section:

	\begin{thm}
	There exists a finite strong stationary time of $X$, whatever the initial distribution, if and only if:
	
		\begin{gather}
		\sum_{i=1}^{\infty}\mu(i+1)\sum_{j=1}^{i}\dfrac{1}{\mu(j)b_j}<\infty\label{borne-droite}\\
		\sum_{i=1}^{\infty}\mu(-i-1)\sum_{j=1}^{i}\dfrac{1}{\mu(-j)a_{-j}}<\infty\label{borne-gauche}
		\end{gather}

	\end{thm}

	\begin{proof}

		\begin{itemize}
		
		\item[$\bullet$] \underline{Necessary condition:}\\
		We consider the same assumption as in Proposition \ref{sharp}, that is $\mu_0=\mu_{|\dbc{-\infty,0}}$ and $X_0^*=\dbc{-\infty,0}$ a.s.. We saw that in this case, a dual $X^*$ with generator $\L^*$ is sharp, so if there exists a strong stationary time then $X^*$ is absorbed in finite time almost surely. But this happens if and only if the series
			$$\sum_{i=1}^{\infty}\prod_{j=1}^{i}\dfrac{\tilde{a}_j}{\tilde{b}_j}\sum_{k=1}^{i}\prod_{l=1}^{k}\dfrac{\tilde{b}_{l-1}}{\tilde{a}_l}$$
		converges. Indeed, this is the criterion for the explosion in finite time of a birth-death process on $\N$ (see e.g. Anderson \cite{anderson}, Section 8.1), that we extend to a process on $\Z$ with values in $\N$ infinitely often, in the same way as for (\ref{vie-mort-Z}). Using the fact that $\mu(\dbc{-\infty,j})$ is uniformly bounded on $j\in\N$, this condition can be restated as follows:

			\begin{align*}
			 &\sum_{i=1}^{\infty}\prod_{j=1}^{i}\dfrac{b_j}{a_{j+1}}\sum_{k=1}^{i}\prod_{l=1}^{k}\dfrac{a_l}{b_l}<\infty\\
			\Leftrightarrow &\sum_{i=1}^{\infty}\mu(i+1)\sum_{j=1}^{i}\dfrac{1}{\mu(j)b_j}<\infty\\
			\end{align*}

		This proves that (\ref{borne-droite}) is a necessary condition. We do the same for (\ref{borne-gauche}), taking $\mu_{|\dbc{0,+\infty}}$ as an initial distribution for $X,\;X^*_0=\dbc{0,+\infty}$ almost surely and considering the left bound of the dual process $X^*$.

		\item[$\bullet$] \underline{Sufficient condition:}\\
		Conversely, let $\mu_0^*:=\sum_{p\in\Z}\mu_0(p)\delta_{\{p\}}$. We have $\mu_0=\mu_0^*\Lambda$ so there exists a $\Lambda$-linked dual $X^*$ of $X$ with generator $\L^*$ and initial distribution $\mu_0^*$. By Lemma \ref{couplage}, we can couple it with a process $(-\infty,Z)$ starting from $\dbc{-\infty,0}$ with generator $\L^*$, so that $(-\infty,Z)\subset X^*$. Thus when $Z$ reaches $+\infty,\; X^*$ is absorbed. This happens in finite time if (\ref{borne-droite}) holds. We proceed similarly for (\ref{borne-gauche}).
			
		\end{itemize}
	
	\end{proof}

\section{Random walk on a graph}\label{sec-marche-graphe}

The previous case may be seen as a random walk on a graph, where $\Z$ is a graph made of two infinite branches connected in $0$. We want to extend the result of the previous section to graphs having finitely many infinite branches, and another finite number of vertices. Such a graph is represented in Figure \ref{graphe}.

	\begin{figure}
	\centering
	\includegraphics{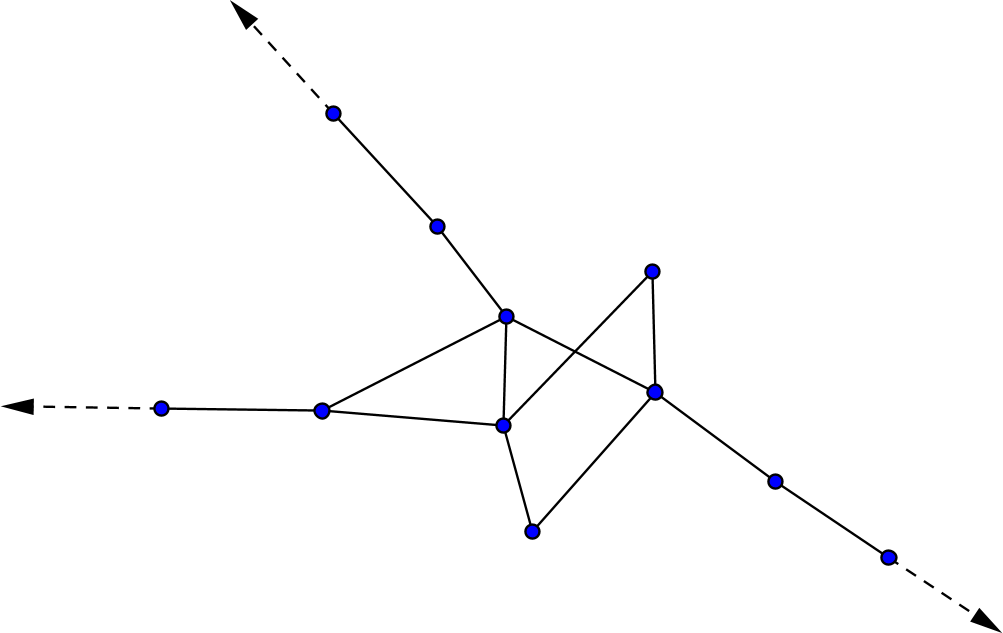}
	\caption{Example of considered graph}
	\label{graphe}
	\end{figure}

Thus we consider a generator $\L$ on a countable set $G$, irreducible and such that for all $x\in G$, the set $\{y\in G:\L_{x,y}>0 \text{ or }\L_{y,x}>0\}$ is finite. Let $(X_t)_{t\geq0}$ be a minimal process with generator $\L$ and call $\mu_0$ its initial distribution. We endow $G$ with the structure of a simple, undirected, weighted and connected graph, for which two distinct vertices $p$ and $q\in G$ are adjacent if and only if $\L_{p,q}\neq0$ or $\L_{q,p}\neq0$ (each vertex then has finite degree).

We assume that $\L$ is nonexplosive, positive recurrent and reversible, so that there exists a unique stationary distribution $\mu$, and we assume in addition that the sum $\sum\mu(p)\L_p$ is finite. We recall that $\mu$ satisfies (\ref{eq-reversible}):
	$$\mu(q)\L_{q,p}=\mu(p)\L_{p,q},\quad\quad\forall p,q\in G$$
and in particular, $\L_{p,q}\neq0$ implies $\L_{q,p}\neq0$. This equality as much as $\underset{q}{\sum}\L_{p,q}=0$ will be the basis of all computations.

Let us start by setting the vocabulary and notations that will be used in the sequel.

	\begin{defi}\label{def-graph}

		\begin{enumerate}[(i)]
		
		\item For every $Q\subset G$, we will write $Q^c$ for the complementary set of $Q$ in $G$.
		
		\item Two vertices of $G$ are \emph{adjacent} if they are connected by an edge.

		For $Q\subset G$ and $q\in Q$, we denote by $\mathrm{V}_Q(q)$ the set of all points of $Q$ adjacent to $q$, $\mathrm{V}(q):=\mathrm{V}_G(q)$ and $\mathrm{V}(Q):=\{q'\in\mathrm{V}_{Q^c}(q), q\in Q\}$.
		
		\item The \emph{degree} of a vertex is the number of its adjacent points.
		
		\item A \emph{path} between two vertices $p,q\in G$ is a sequence of vertices $p=p_1,\dots,p_n=q$ such that $p_i$ and $p_{i+1}$ are adjacent for every $1\leq i<n$, and that all $p_i$ are distinct.
		
		\item A subset $Q$ of $G$ is \emph{connected} if between any two vertices of $Q$ there exists a path in $Q$. It is convex if any path between two vertices of $Q$ is in $Q$.
		
		\item If $Q'\subset Q\subset G$, $Q'$ is a \emph{connected component} of $Q$ if $Q'$ is connected and maximal (for set inclusion) among connected subsets of $Q$. We write $\mathscr{C}(Q)$ for the set of all connected components of $Q$.
		
		\item We call \emph{center of G} and denote by $C_G$ the union of all paths between any two vertices with degrees different from two (namely, the \emph{convex envelope} of these vertices).
		
		\item We call \emph{infinite branch} any connected component (necessarily infinite) of $G\setminus C_G$. We denote $Q_i,\ i\in I$ these infinite branches of $G$.
		
		\item To each infinite branch $Q_i$, we add an infinite point, written $\Delta_i$. We write $\bar{Q}_i=Q_i\cup\{\Delta_i\}$.
		
		\item In the case where $C_G$ is non empty, if we endow $\N$ with a structure of graph for which two points $n$ and $m\in\N$ are adjacent if and only if $|n-m|=1$, then each $Q_i$ is isomorphic as a graph to $\N$ and we call $\varphi_i\colon\N\cup\{\infty\}\rightarrow \bar{Q}_i$ the unique corresponding isomorphism, completed by setting $\varphi_i(\infty)=\Delta_i$. We will also use the following notation:
			$$\varphi_i(\dbc{p,\infty}):=\underset{q\geq p}{\bigcup}\{\varphi_i(q)\}\cup\{\Delta_i\}$$
		
		\item We write as in (\ref{def-mu^i}):
			\begin{equation*}
			\mu^i(n):=\prod_{k=0}^{n-1}\dfrac{\L_{\varphi_i(k),\varphi_i(k+1)}}{\L_{\varphi_i(k+1),\varphi_i(k)}}=\dfrac{\mu(\varphi_i(k))}{\mu(\varphi_i(0))}
			\end{equation*}
			
		\end{enumerate}

	\end{defi}

We now assume that $C_G$ is finite, as well as the number of infinite branches ($I=\dbc{1,\dots,N}$ with $N\in\N$), and non empty (the case  where $C_G$ is empty corresponds to $G=\Z$, which has been treated in the previous section). With those assumptions, we set:
	$$\bar{G}=G\cup\bigcup_{i\in I}\{\Delta_i\}$$
and we define a distance $\bar{\mathrm{d}}$ on $\bar{G}$ as follows:

	\begin{itemize}
	
	\item If $p\in Q_i,\ q\in\bar{Q_i},\ i\in I$ with $\varphi_i^{-1}(p)\leq \varphi_i^{-1}(q)$: 
		$$\bar{\mathrm{d}}(q,p)=\bar{\mathrm{d}}(p,q)=\sum_{n=\varphi_i^{-1}(p)+1}^{\varphi_i^{-1}(q)}\dfrac{1}{n^2}$$
	
	\item If $p,q\in C_G$:
		$$\bar{\mathrm{d}}(p,q)=\inf\{\mathrm{card}(\gamma),\gamma \text{ path between } p\text{ and } q\}-1$$
	that is, the length of the shortest path (in particular, if $p$ and $q$ are adjacent, then $\bar{\mathrm{d}}(p,q)=1$).
	
	\item If $i\in I$ and $(p_i,q_i)$ is the unique adjacent pair in $Q_i\times C_G$, we set:
		$$\bar{\mathrm{d}}(q,p)=\bar{\mathrm{d}}(p,q):=1+\bar{\mathrm{d}}(p,p_i)+\bar{\mathrm{d}}(q_i,q),\ \forall p\in \bar{Q_i},\;\forall q\in C_G$$
	and:
		$$\bar{\mathrm{d}}(q,p)=\bar{\mathrm{d}}(p,q):=\bar{\mathrm{d}}(p,q_i)+\bar{\mathrm{d}}(q_i,q),\ \forall p\in \bar{Q_i}\,\ \forall q\in \bar{Q_j}\ (i\neq j)$$
	
	\end{itemize}

	\begin{rem}\label{rm-chemin}
	If $p_1,p_2,p_3\in Q_i,\ i\in I$, are three vertices of the same infinite branch, then $\bar{d}(p_1,\Delta_i)<\bar{d}(p_2,\Delta_i)<\bar{d}(p_3,\Delta_i)$, if and only if the unique path between $p_1$ and $p_3$ contains $p_2$.
	\end{rem}

One easily checks that $\bar{G}$ endowed with this metric is compact: the induced topology on each branch is that of the Alexandrov compactification.

Let $G^*$ be the set of all non empty connected compact subsets of $\bar{G}$, endowed with the Hausdorff distance:
	$$\mathrm{d}^*(Q,Q'):=\max\{\underset{p\in Q}{\sup}\;\underset{p'\in Q'}{\inf}\mathrm{\bar{d}}(p,p'),\ \underset{p'\in Q'}{\sup}\;\underset{p\in Q}{\inf}\mathrm{\bar{d}}(p,p')\}$$
It is a closed subspace of the space of compact subsets of $\bar{G}$ (endowed with the Hausdorff distance). We recall that the latter is compact, thus so is $G^*$.

\paragraph{}Let us describe briefly those two metric spaces.\\
Denoting by $B_{\mathrm{d}}(x,\varepsilon)$ (respectively $\bar{B}_{\mathrm{d}}(x,\varepsilon)$) the open (respectively closed) ball with center $x$ and radius $\varepsilon$ for a distance $\mathrm{d}$, we have:
	$$B_{\bar{\mathrm{d}}}(p,1/2)=\bar{B}_{\bar{\mathrm{d}}}(p,1/2)=\{p\},\quad\forall p\in C_G$$
	$$B_{\bar{\mathrm{d}}}(\varphi_i(p),\dfrac{1}{(p+1)^2})=\bar{B}_{\bar{\mathrm{d}}}(\varphi_i(p),\dfrac{1}{2(p+1)^2})=\{\varphi_i(p)\},\quad\forall p\in \N,\;\forall i\in I$$
	$$B_{\bar{\mathrm{d}}}(\Delta_i,\varepsilon)=\{\varphi_i(n):\;\sum_{j=n+1}^{+\infty}\dfrac{1}{j^2}<\varepsilon\}\quad\forall i\in I,\;\forall \varepsilon<\sum_{j=0}^{+\infty}\dfrac{1}{j^2}$$
	$$\bar{B}_{\bar{\mathrm{d}}}(\Delta_i,\varepsilon)=\{\varphi_i(n):\;\sum_{j=n+1}^{+\infty}\dfrac{1}{j^2}\leq\varepsilon\}\quad\forall i\in I,\;\forall \varepsilon<\sum_{j=0}^{+\infty}\dfrac{1}{j^2}$$
The sets above thus are the elementary open sets for the topology induced by the distance $\bar{\mathrm{d}}$. In particular, each singleton is closed (being the complementary set of a countable number of open sets), and each singleton of $G$ is also open. In addition, the $\Delta_i$ are the only limit points of $\bar{G}$.

The elements of $G^*$ are compact (hence closed) subsets of $\bar{G}$, that is, finite unions of singletons of $G$ and of neighborhoods of some $\Delta_i$ such as described earlier. 

The previous statements are easy to check from definition. The following result is not much harder:

	\begin{Lem}\label{voisinages}
	Let $Q\in G^*$, and $I_0:=\{i\in I:\;\Delta_i\in Q \}$. Then, for all $\eta>0$, there exists $\varepsilon>0$ such that for all $Q'\in B_{\mathrm{d}^*}(Q,\varepsilon)$:
		$$Q'\setminus\left(\underset{i\in I_0}{\bigcup}B_{\bar{\mathrm{d}}}(\Delta_i,\eta)\right)= Q\setminus\left(\underset{i\in I_0}{\bigcup}B_{\bar{\mathrm{d}}}(\Delta_i,\eta)\right)$$
	\end{Lem}
In other words, a "sufficiently close" neighbor of $Q$ must coincide with $Q$ outside the neighborhoods of the $\Delta_i\in G$. In particular, if $Q\in G^*$ is included in $G$, then $Q$ is an isolated point (in the sense that $\{Q\}$ is open).

	\begin{proof}
	One only needs to take:
	$$\epsilon<\inf\{\mathrm{\bar{d}}(p,q):\;p,q\notin\underset{i\in I_0}{\bigcup}B_{\bar{\mathrm{d}}}(\Delta_i,\eta),\;p\in Q,\; q\notin Q\}$$
	\end{proof}

\paragraph{}For every $i\in I\cup\{0\}$, we set:
	\begin{equation}
	G^*_i:=\{Q\in G^*:\;|\{j\in I:\;\Delta_j\in Q\}|=i \}\label{set-G_i}
	\end{equation}
the set of all non empty connected compact subsets of $G$ which contains exactly $i$ points $\Delta_j$.

\paragraph{}We define a generator $\L^*$ on $G^*$ by:

	\begin{equation}
	\forall Q\neq Q'\in G^*,\quad \L^*_{Q,Q'}:=
		\begin{cases}
		\left(\dfrac{\mu(Q')}{\mu(Q)}\right)\underset{q\in Q}{\sum}\L_{p,q} & \text{ if } Q'=Q\cup\{p\},\ p\notin Q\smallskip\\
		\left(\dfrac{\mu(Q')}{\mu(Q)}\right)\underset{q\notin Q}{\sum}\L_{p,q} & \text{ if } Q'\in\mathscr{C}(Q\setminus\{p\}),\ p\in Q\\
		0 & \text{ otherwise}
		\end{cases}
	\end{equation}
	
\noindent (we recall that $\mathscr{C}(Q)$ is the set of all connected components of $Q$).
	
We observe that the above coefficients are well defined, since if $Q'\in\mathscr{C}(Q\setminus\{p\})\cap\mathscr{C}(Q\setminus\{q\})$ with $p,q\in Q$, then necessarily $p,q\notin Q'$ and $Q'\cup\{p\},\;Q'\cup\{q\}$ are connected, so $p=q$.
	
The dynamic of a minimal process $X^*$ with generator $\L^*$ should be understood this way: at each jumping time, we add to $X^*$ an adjacent point, or we remove a point that is adjacent to its complementary set. When $X^*$ disconnects, we choose one of its connected components with a probability proportional to its weight relative to $\mu$.

From this, we deduce that the number of possible transitions is bounded: given $Q\in G^*$, the number of $Q\cup\{p\}$ with $p\in V(Q)\setminus Q$ is bounded by $N+|C_G|$ and the number of $Q'\in \mathscr{C}(Q\setminus\{p\})$, with $p\in V(Q^c)\cap Q$ is bounded (roughly) by $|C_G|^2$.

	\begin{Lem}
	$\L^*$ is a stratified generator on $G^*$, and $(G^*_0,\dots,G^*_N)$ is a stratification of $\L^*$.
	\end{Lem}

	\begin{proof}The $G^*_i$ are disjoint and the coefficients of $\L^*$ prevent any transition from some set $G^*_i$ toward some $G^*_j$, if $j>i$. The first two conditions of Definition \ref{def-strat} are then fulfilled.
  	
	\paragraph{}Let $X^*$ be a trajectory of a minimal process with generator $\L^*$. We call $T_n,\;n\in\N$, its jumping times, $(Y_n)_{n\in\N}$ its jump chain and $T=\lim T_n$ its explosion time, that we assume to be finite. Thus $Y$ passes through infinitely many states, and visits each of them finitely many times. Since $G^*$ is compact, it follows that $Y$ has at least one limit point, which is a limit point of $G^*$. If it is unique, then $Y$ converges to this point. Let us assume by contradiction that there are two distinct limit points, that we call $E_1$ and $E_2$, respective limits of two subsequences $(Y_{\phi_1(n)})_n$ and $(Y_{\phi_2(n)})_n$ of $(Y_n)_n$. The idea of what follows is to show that we cannot go infinitely many times from one subsequence to the other one.
  	
	Since $E_1$ are $E_2$ two connected compact subsets of $\bar{G}$ and $E_1\neq E_2$, there exists $p\in G$ such that $p\in E_1$ and $p\notin E_2$ (or the other way round).
  	  	
	\paragraph{}From Lemma \ref{voisinages}, there exist $n_1,n_2\in\N$ such that $p\in Y_{\phi_1(n)},\ \forall n>n_1$ and $p\notin Y_{\phi_2(n)},\ \forall n>n_2$. After the time $\max\{\phi_1(n_1),\phi_2(n_2)\}$ the chain $Y$ jumps infinitely many times from a state $Q$ containing an adjacent point of $p$ toward the state $Q\cup\{p\}$. This is impossible almost surely, the rate of such a transition being upper bounded by:
		$$\left(1+\underset{q\in V(p)}{\sup}\dfrac{\mu(p)}{\mu(q)}\sum_{q'\in V(p)}\L_{p,q'}\right)<\infty$$
		
	\paragraph{}Thus, the sequence $(Y_n)_n$ has an almost surely unique limit $E\in G^*_k,\; k\in I\cup\{0\}$. Let $I_0:=\{i\in I:\;\Delta_i\in E \}$. From Lemma \ref{voisinages}, for all $\eta>0$, there exists $n_1$ such that for every $n>n_1$:
		$$Y_n\setminus\left(\underset{i\in I_0}{\bigcup}B_{\bar{\mathrm{d}}}(\Delta_i,\eta)\right)=E\setminus\left(\underset{i\in I_0}{\bigcup}B_{\bar{\mathrm{d}}}(\Delta_i,\eta)\right)$$
	In particular, if $i\notin I_0$, then $\Delta_i\notin Y_n$. In addition, there exists $n>n_1$ such that $Y_n\neq E$, and since $Y_n$ is connected, there exists $i\in I_0$ such that $\Delta_i\notin Y_n$. There are therefore less $\Delta_i$ in $Y_n$ than in $E$, that is, $Y_n\in G^*_j$ with $j<k$.
	
	Now, the number of transitions of the minimal process from some state in $G^*_i$ toward some other one in $G^*_j,\; i\neq j$, is finite a.s. (since they are only possible if $j<i$) and thus there exists a last one before the explosion. Let $T_{n_2}$ be the time of this last transition, and $G^*_j,\;j<k,$ the arrival set ($Y_{n_2}\in G^*_j$), then, for all $t\in\left[T_{n_2},T\right[,\;X_t^*\in G^*_j$. So the third condition is also satisfied.
	\end{proof}

\paragraph{}We define the kernel $\Lambda(\cdot,\cdot)$ by:
	$$\forall Q\in G^*,\ p\in G,\quad \Lambda(Q,p)=\dfrac{\mu(p)\delta_p(Q)}{\mu(Q)}$$
$\Lambda$ is continuous in the first variable, and:

	\begin{Lem}
	$\L$ and $\L^*$ satisfy $\L^*\Lambda=\Lambda\L$.
	\end{Lem}

	\begin{proof}
	Let $Q\in G^*$ and $p\in G$. If $p\notin Q$, then:

		\begin{align*}
		\sum_{Q'\in G^*}\L^*_{Q,Q'}\Lambda(Q',p)&=\sum_{q\in G}\L^*_{Q,Q\cup \{q\}}\dfrac{\mu(p)\delta_p(Q\cup q)}{\mu(Q\cup\{q\})}\\
		&=\L^*_{Q,Q\cup \{p\}}\dfrac{\mu(p)}{\mu(Q\cup\{p\})}\\
		&=\dfrac{\mu(p)}{\mu(Q)}\sum_{q\in Q}\L_{p,q}\\
		&=\sum_{q\in Q}\dfrac{\mu(q)}{\mu(Q)}\L_{q,p}\\
		&=\sum_{q\in Q}\Lambda(Q,q)\L_{q,p}
		\end{align*}

	On the other hand, if $p\in Q$, then:

		\begin{align*}
		\sum_{Q'\in G^*}\L^*_{Q,Q'}\Lambda(Q',p)&=\sum_{q\notin Q}\L^*_{Q,Q\cup\{q\}}\dfrac{\mu(p)}{\mu(Q\cup\{q\})}+\L^*_{Q,Q}\Lambda(Q,p)\\
		&\quad\quad+\sum_{q\neq p\in Q}\;\sum_{Q'\in\mathscr{C}(Q\setminus\{q\})}\L^*_{Q,Q'}\dfrac{\mu(p)}{\mu(Q\setminus\{q\})}\\
		&=\sum_{q\notin Q}\sum_{q'\in Q}\dfrac{\mu(p)}{\mu(Q)}\L_{q,q'}+\sum_{q\neq p\in Q}\sum_{q'\notin Q}\dfrac{\mu(p)}{\mu(Q)}\L_{q,q'}\\
		&\quad\quad+\dfrac{\mu(p)}{\mu(Q)}\left(\underset{q\notin Q}{\sum}\underset{q'\notin Q}{\sum}\L_{q,q'}+\underset{q\in Q}{\sum}\underset{q'\in Q}{\sum}\L_{q,q'}\right)\\
		&=\dfrac{\mu(p)}{\mu(Q)}\left(\sum_{q\notin Q, q'\in G}\L_{q,q'}+\sum_{q\neq p\in Q,q'\in G}\L_{q,q'}+\sum_{q'\in Q}\L_{p,q'}\right)\\
		&=\sum_{q\in Q}\dfrac{\mu(p)}{\mu(Q)}\L_{p,q}\\
		&=\sum_{q\in Q}\dfrac{\mu(q)}{\mu(Q)}\L_{q,p}
		\end{align*}

	\noindent We then have proved that $(\L^*\Lambda)_{Q,p}=(\Lambda\L)_{Q,p}$.
	\end{proof}

\paragraph{}Thanks to the previous two lemmas and Theorem \ref{cond-dual2} we deduce that whatever the initial distribution $\mu_0$ of $X$, there exists a stratified jump process $X^*$ with generator $\L^*$, which is $\Lambda$-linked to $X$ (taking for an initial distribution of $X^*$, for example, $\mu^*_0(A)=\underset{x\in\S}{\sum}\mu_0(x)\ind_{A=\{x\}}$ for all $A\in\S^*$). In that case, by definition of $\Lambda$, $X^*$ is also a strong stationary dual of $X$.
	
One will have noticed that a process $X^*$ with generator $\L^*$ has $N+1$ absorbing states: each of the singletons $\{\Delta_i\}$ and $\bar{G}$. However, by intertwining $X$ and $X^*$ with $\Lambda$, we get:
	\begin{align}
	\nonumber 0=&\P(X_t=\Delta_i\ |\ X^*_0\neq\{\Delta_i\})\\
	\nonumber \geq&\P(X_t=\Delta_i\ |\ X^*_t=\{\Delta_i\},\ X^*_0\neq\{\Delta_i\})\;\P(X^*_t=\{\Delta_i\}\ |\ X^*_0\neq\{\Delta_i\})\\
	\nonumber =&\P(X_t=\Delta_i\ |\ X^*_t=\{\Delta_i\})\;\P(X^*_t=\{\Delta_i\}\ |\ X^*_0\neq\{\Delta_i\})\\
	=&\P(X^*_t=\{\Delta_i\}\ |\ X^*_0\neq\{\Delta_i\})\label{absorbant}
	\end{align}
since $X$ is nonexplosive. We conclude that the only absorbing state that may be reached in finite time is $\bar{G}$.

\paragraph{}It remains to study the time to absorption of $X^*$. More precisely, we will study the behavior of the bounds of $X^*$ on each infinite branch, by comparing them with birth-death processes on $\N$.
	
For every $i\in I$, and every $p\in\N$, we set:
	\begin{equation}
	G^i_p:=\bar{G}\setminus \varphi_i(\dbc{p+1,\infty})=\underset{j\neq i}{\bigcup}\bar{Q}_j\cup C_G\cup\varphi_i(\dbc{0,p})\in G^*\label{set-G-p}
	\end{equation}
(that is, the point of $G^*$ containing all $\bar{G}$ except the points of $\bar{Q}_i$ beyond $\varphi_i(p)$), and:
	$$\begin{cases}
	\L^i_{p,p+1}:=\left(\dfrac{\mu(G^i_{p+1})}{\mu(G^i_p)}\right)\L_{\varphi_i(p+1),\varphi_i(p)}=\L^*_{G^i_p,G^i_{p+1}}\medskip\\
	\L^i_{p+1,p}:=\left(\dfrac{\mu(G^i_{p})}{\mu(G^i_{p+1})}\right)\L_{\varphi_i(p+1),\varphi_i(p+2)}=\L^*_{G^i_{p+1},G^i_{p}}\\
	\L^i_{p,p}=-\L^i_{p,p+1}-\L^i_{p,p-1}
	\end{cases}$$
which define generators $(\L^i_{p,q})_{p,q\in\N}$ of birth-death processes on $\N$.

	\begin{Lem}\label{couple-suffisante}
	There exists $\lambda>0$ such that, for every $i\in I$ and every probability measure $\mu_0^*$ on $G^*$, there exists a family $(X^*,Y,Z^i)$ such that:
	
		\begin{itemize}
		\item The couple $(X^*,Z^i)$ is Markovian, and its marginals $X^*$ and $Z^i$ are Markov processes, with respective generators $\L^*$ and $\L^i$.
		\item $Y$ is an exponential random variable with parameter $\lambda$, independent from $Z^i$.
		\item $X^*$ has initial distribution $\mu_0^*$.
		\item For every $t>s>0$:
			$$\P(X^*_t\cap\varphi_i(\dbc{Z^i_t,\infty})\neq\varnothing\ |\ t<Y,\;X^*_0\nsubseteq G^i_0)=1$$
			$$\P(X^*_s\cap C_G\subset X^*_t\ |\ t<Y)=1$$
		\end{itemize}
	
	\end{Lem}

	\begin{proof}
	We are first going to distinguish different kinds of transition for the process $X^*$, when it is at a state $Q\in G^*$ such that $Q\cap Q_i\ne\varnothing$. The first kind is that which changes $X^*$ only by a point of $Q_i$. It concerns the transitions from $Q$ to $Q\cup\{\varphi_i(p)\}$, or to $Q\setminus\{\varphi_i(p-1)\}$, where $p=\sup\{q\in\N:\;\varphi_i(q)\in Q\}+1$. These transitions "make us one step closer to or further from" a state containing the point $\Delta_i$. We will thus denote $A^+_i(Q)=Q\cup\{\varphi_i(p)\}$, $A^-_i(Q)=Q\setminus\{\varphi_i(p-1)\}$ and $A_i(Q)=\{A^+_i(Q),\;A^-_i(Q)\}$.
	
	The second kind of transitions may remove several points at once. To do so, there must exist $q\in Q\in G^*$ such that $Q\setminus\{q\}$ is disconnected and $q\in V(Q^c)$. This is only possible if $q\in C_G$. Thus, we naturally set:
		$$B(Q):=\{Q'\in\mathscr{C}(Q\setminus\{q\}):\;q\in Q\cap C_G\}$$
		
	The third kind gathers all other transitions, those which do not interfere in our problematic:
		$$C_i(Q)=\{Q'\in G^*:\L_{Q,Q'}>0,\;Q'\notin A_i(Q)\cup B(Q)\}$$
	Those transitions remove a point of $Q_j,\;j\neq i$ or add a point different from $\varphi_i(p)$.
	
	We will talk about a $A^+_i$ (resp. $A^-_i$, resp. $B$, resp. $C$) transition for a transition from $Q$ to some $Q'=A^+_i(Q)$ (resp $Q'=A^-_i$, resp. $Q'\in B(Q)$, resp. $Q'\in C_i(Q)$).
	
	\paragraph{}We observe that the set $\{\L^*_{Q,Q'}:Q\in G^*,\;Q'\in B(Q)\}$ is bounded. Indeed, if $Q\in G^*$ is such that $B(Q)$ is non empty, then it contains a point of $C_G$, so:
		$$\mu(Q)\geq\inf\{\mu(p):\;p\in C_G\}=:a_1$$
	and if $p_0\in Q\cap C_G$, then:
		$$\sum_{q\notin Q}\L_{q,p_0}\leq\underset{p\in C_G}{\sup}\sum_{q\in V(C_G)}\L_{q,p}=:a_2$$
	which gives us for all $Q'\in \mathscr{C}(Q\setminus\{p_0\})$:
	
		\begin{align*}
		\L^*_{Q,Q'}&=\dfrac{\mu(Q')}{\mu(Q)}\sum_{q\notin Q}\L_{q,p_0}\\
		&\leq \dfrac{1}{a_1}a_2
		\end{align*}
		
	The set $\{|B(Q)|:\;Q\in G^*\}$ is also bounded by a constant $a_3$ (roughly, $N+|C_G|$), so we deduce:
		$$\sum_{Q'\in B(Q)}\L^*_{Q,Q'}\leq \dfrac{a_2a_3}{a_1}=:\lambda$$
	and $\lambda$ does not depend on $Q$.
	
	We finally observe that for every $p\in\N$, if $\varphi_i(p)\in Q$ and $\varphi_i(p+1)\notin Q,\; Q\in G^*$, then $\L^*_{Q,A^+_i(Q)}\geq \L^i_{p,p+1}$ and $\L^*_{Q,A^-(Q)}\leq \L^i_{p,p-1}$.
	
	\paragraph{}We are now going, thanks to the above observations, to construct the family $(X^*,Y,Z^i)$ by preventing any $B$ transition before time $Y$, and by forcing $Z^i$ to go down in case of a $A^-_i$ transition or by setting a $A^+_i$ transition if $Z^i$ goes up, when necessary. Of course, this construction is only made when $X^*_0$ contains at least one point of $Q_i$.
	
	\paragraph{Construction of the family $(X^*,Y,Z^i)$:}
	In order to be reduced to a Markov process, we will construct instead $(X^*,N,Z^i)$, where $N$ is a Poisson process with intensity $\lambda$. $Y$ is then the first jumping time of $N$. Firstly, we draw $X^*_0$ with distribution $\mu^*_0$. If $X^*_0\cap Q_i=\varnothing$, we can stop the construction. Otherwise, we set $Z^i_0=\sup\{p\in\N:\varphi_i(p)\in X^*_0\}$ and $T_0=0$.
	
	We assume $(X^*,N,Z^i)$ to be constructed up to the time $T_n,\;n\geq 0$, $Z^i_{T_n}=p$ and $X^*_{T_n}=Q$. If $N_t\geq1$, the construction is over and we let $X^*$ and $Z$ evolve independently. Otherwise, we separate two cases:
	
		\begin{itemize}
		\item \emph{First case: $\varphi_i(Z^i_{T_n})\notin A^-_i(X^*_{T_n})$}

		This is the case where $Z^i_{T_n}$ is "at the far end" of $X^*_{T_n}$, and at this point we must be careful to couple $Z^i$ and $X^*$ so that $Z^i$ stays "inside" at time $T_{n+1}$. We thus draw an exponential variable $\epsilon^1_n$ with parameter $\lambda$ and another one, $\epsilon^2_n$, independent from $\epsilon^1_n$ and with parameter:
			$$\Lambda^2_n:=\sum_{Q'\in C_i(Q)}\L^*_{Q,Q'}+\L^*_{Q,A_i^+(Q)}+\L^i_{p,p-1}$$
		We set $T_{n+1}=T_n+\epsilon^1_n\wedge\epsilon^2_n$. If $\epsilon^1_n<\epsilon^2_n$, we enforce $N_{T_{n+1}}=1$ and choose $X^*_{T_{n+1}}=Q'$ with probability $\dfrac{\L^*_{Q,Q'}}{\lambda}$, for all $Q'\in B(Q)$. The sum of these probabilities may be less than $1$, so if no state is chosen we stay in $Q$.
		
		If $\epsilon^1_n>\epsilon^2_n$, we leave $N_{T_{n+1}}$ at $0$, and we set:
			$$(X^*_{T_{n+1}},Z^i_{T_{n+1}})=\begin{cases}
			(Q',p)\quad\text{with probability }\dfrac{\L^*_{Q,Q'}}{\Lambda^2_n},\;Q'\in C_i(Q)\smallskip\\
			(Q',p+1)\quad\text{with probability }\dfrac{\L^i_{p,p+1}}{\Lambda^2_n},\;Q'=A^+_i(Q)\smallskip\\
			(Q',p)\quad\text{with probability }\dfrac{\L^*_{Q,Q'}-\L^i_{p,p+1}}{\Lambda^2_n},\;Q'=A^+_i(Q)\\
			(Q,p-1)\quad\text{with probability }\dfrac{\L^i_{p,p-1}-\L^*_{Q,A^-_i(Q)}}{\Lambda^2_n}\\
			(Q',p-1)\quad\text{with probability }\dfrac{\L^*_{Q,A^-_i(Q)}}{\Lambda^2_n},\;Q'=A^-_i(Q)\\			
			\end{cases}$$
			
		\item \emph{Second case: $\varphi_i(Z^i_{T_n})\in A^-_i(X^*_{T_n})$}
		
		In that case, we do not need to ensure that "the far end" of $X^*$ stays above $Z^i$, but only to check if a $B$ transition occurs. Consequently, we draw a exponential variable $\epsilon^1_n$ with parameter $\lambda$ and another one, $\epsilon^2_n$, independent from $\epsilon^1_n$ and with parameter:
			$$\Lambda^2_n:=\sum_{Q'\in C_i(Q)\cup A_i(Q)}\L^*_{Q,Q'}+\L^i_{p,p+1}+\L^i_{p,p-1}$$
		We set $T_{n+1}=T_n+\epsilon^1_n\wedge\epsilon^2_n$. If $\epsilon^1_n<\epsilon^2_n$, we proceed exactly as in the first case. Otherwise, we chose:
			$$(X^*_{T_{n+1}},Z^i_{T_{n+1}})=\begin{cases}
			(Q',p)\quad\text{with probability }\dfrac{\L^*_{Q,Q'}}{\Lambda^2_n},\;Q'\in C_i(Q)\cup A_i(Q)\smallskip\\
			(Q,p+1)\quad\text{with probability }\dfrac{\L^i_{p,p+1}}{\Lambda^2_n}\smallskip\\
			(Q,p-1)\quad\text{with probability }\dfrac{\L^i_{p,p-1}}{\Lambda^2_n}
			\end{cases}$$
		By construction, we do have $X^*_t\cap\varphi_i(\dbc{Z^i_t,\infty})\neq\varnothing$ if $t<Y$, and $Y$ and $Z^i$ independent, until the first explosion time $\tau_1:=\lim T_n$. One checks that the marginal processes have the desired generators. Starting from the explosion time, we carry on the construction the same way.
		\end{itemize}
		
	\end{proof}

	\begin{Lem}\label{condition-suffisante}
	For all $i\in I$ if $Z^i$ is a process with generator $\L^i$, then the explosion time of $Z^i_t$ is finite almost surely if and only if the condition (\ref{borne-infini}):
		\begin{equation*}
		\sum_{j=1}^{\infty}\mu^i(j+1)\sum_{k=1}^{j}\dfrac{1}{\mu^i(k)\L_{\varphi_i(k),\varphi_i(k+1)}}<\infty
		\end{equation*}
	holds. In that case, $Z^i$ goes to $+\infty$ in finite time, otherwise it is nonexplosive.
	\end{Lem}

	\begin{proof}
	We recall that the explosion criterion for $Z^{i}$ is:
		$$\sum_{i=1}^{\infty}\prod_{j=1}^{i}\dfrac{\L^{i}_{j,j-1}}{\L^{i}_{j,j+1}}\sum_{k=1}^{i}\prod_{l=1}^{k}\dfrac{\L^{i}_{j-1,j}}{\L^{i}_{j,j-1}}<\infty$$
	(cf \cite{anderson}, Section 8.1). That is:
		
		\begin{align*}
		\infty&>\sum_{n=1}^{\infty}\prod_{j=1}^{n}\dfrac{\left(\dfrac{\mu(G^i_{j-1})}{\mu(G^i_{j})}\right)\L_{\varphi_i(j),\varphi_i(j+1)}}{\left(\dfrac{\mu(G^i_{j+1})}{\mu(G^i_j)}\right)\L_{\varphi_i(j+1),\varphi_i(j)}}\sum_{k=1}^{n}\prod_{l=1}^{k}\dfrac{\left(\dfrac{\mu(G^i_j)}{\mu(G^i_{j-1})}\right)\L_{\varphi_i(j),\varphi_i(j-1)}}{\left(\dfrac{\mu(G^i_{j-1})}{\mu(G^i_{j})}\right)\L_{\varphi_i(j),\varphi_i(j+1)}}\\
		&\geq\sum_{n=1}^{\infty}\prod_{j=1}^{n}\dfrac{\mu(G^i_{j-1})\L_{\varphi_i(j),\varphi_i(j+1)}}{\mu(G^i_{j+1})\L_{\varphi_i(j+1),\varphi_i(j)}}\sum_{k=1}^{n}\prod_{l=1}^{k}\dfrac{\L_{\varphi_i(j),\varphi_i(j-1)}}{\L_{\varphi_i(j),\varphi_i(j+1)}}\\
		&\geq\sum_{n=1}^{\infty}\prod_{j=1}^{n}\dfrac{\mu(G^i_{0})\L_{\varphi_i(j),\varphi_i(j+1)}}{\L_{\varphi_i(j+1),\varphi_i(j)}}\sum_{k=1}^{n}\prod_{l=1}^{k}\dfrac{\L_{\varphi_i(j),\varphi_i(j-1)}}{\L_{\varphi_i(j),\varphi_i(j+1)}}\\
		&\geq\dfrac{\mu(G^i_{0})}{\mu(\varphi_i(0))}\sum_{n=1}^{\infty}\mu(\varphi_i(n+1))\sum_{k=1}^{n}\dfrac{1}{\mu(\varphi_i(k))\L_{\varphi_i(k),\varphi_i(k+1)}}\\
		&=\dfrac{\mu(G^i_{0})}{\mu(\varphi_i(0))}\sum_{n=1}^{\infty}\mu^i(n+1)\sum_{k=1}^{n}\dfrac{1}{\mu^i(k)\L_{\varphi_i(k),\varphi_i(k+1)}}\\
		\end{align*}
	
	\noindent where the $G^i_j$ are the subsets of $\bar{G}$ defined in (\ref{set-G-p}).
	\end{proof}
%
%
	
The purpose of these generators $\L^i$ is then to provide a sufficient condition for the explosion of a process $X^*$ with generator $\L^*$ along each branch $Q_i$. We get a necessary condition similarly, defining, for every $i\in I$ and every $n\in\N$ and $p\geq n$:

	$$\begin{cases}
	\L^{i,n}_{p,p+1}:=\left(1+\dfrac{\mu(\varphi_i(p+1))}{\mu(\varphi_i([n,p]))}\right)\L_{\varphi_i(p+1),\varphi_i(p)}=\L^*_{\varphi_i([n,p]),\varphi_i([n,p+1])}\smallskip\\
	\L^{i,n}_{p,p-1}:=\left(1-\dfrac{\mu(\varphi_i(p))}{\mu(\varphi_i([n,p]))}\right)\L_{\varphi_i(p-1),\varphi_i(p)}=\L^*_{\varphi_i([n,p]),\varphi_i([n,p-1])}\\
	\L^{i,n}_{p,p}=-\L^{i,n}_{p,p+1}-\L^{i,n}_{p,p-1}
	\end{cases}$$

The generators $(\L^{i,n}_{p,q})_{p,q\geq n}$ define birth-death processes on $[n,+\infty]$. Exactly as above, we get the following lemmas:

	\begin{Lem}\label{couple-necessaire}
	For every $i\in I,$ and every probability measure $\mu_0^*$ on $G^*$, there exists a process $X^*$ with generator $\L^*$ and initial distribution $\mu_0^*$ and a sequence $(Z^{i,n})_{n\in\N}$ of birth death processes such that, for every $n\in\N$, $Z^{i,n}$ has generator $\L^{i,n}$ and:
		$$\P(X^*_t\cap\varphi_i(\dbc{Z^{i,n}_t,\infty})\subset\{\varphi_i(Z^{i,n}_t)\}\ |\ \Delta_i\notin X^*_0,\;t<T_{n})=1$$
	for all $t>0$, with $T_{n}:=\inf\{s>0:\;X^*_s\cap G_{n}^i=\varnothing\}$.
	\end{Lem}

	\begin{proof}
	We only need to adapt slightly the proof of Lemma \ref{couplage}. This time, we let $X^*$ and each of the $Z^{i,n}$ evolve independently up to the time $\tau:=\inf\{t>0:\;\exists n\in\N, Z^{i,n}\in X^*_t\}$, then we apply the same method to force $X^*$ to stay in $G^i_{Z^{i,n}}$.
	\end{proof}

	\begin{Lem}\label{condition-necessaire}
	If $Z^{i,n}$ is a process with generator $\L^{i,n}$, with $i\in I,\; n\in\N$ then the explosion time of $Z^{i,n}$ is finite almost surely if and only if:
		\begin{equation}\label{borne-infini-bis}
		\sum_{j=n+1}^{\infty}\mu^i(j+1)\sum_{k=n+1}^{j}\dfrac{1}{\mu^i(k)\L_{\varphi_i(k),\varphi_i(k+1)}}<\infty
		\end{equation}
	\end{Lem}

	\begin{thm}
	Let $X^*$ be a process with generator $\L^*$. If there exists $i\in I$ such that the condition (\ref{borne-infini}) is not fulfilled, then for all $t>0$:
		$$\P(\Delta_i\in X^*_t\ |\ \Delta_i\notin X^*_0)=0$$
	Conversely, if for every $i\in I$ the condition (\ref{borne-infini}) is fulfilled then, almost surely, there exists $T<\infty$ such that $X^*_t=\bar{G},\;\forall t>T$.
	\end{thm}

	\begin{proof}
	We prove the first part with Lemmas \ref{couple-necessaire} and \ref{condition-necessaire}. Let $i\in I$ such that the condition (\ref{borne-infini}) is not fulfilled. We see immediately that for every $n\in\N$, the condition (\ref{borne-infini-bis}) is not either. Let then $(\tilde{X}^*,(Z^{i,n})_{n\in\N})$ be a family of processes as in Lemma \ref{couple-necessaire}, $\tilde{X}^*$ having same law as $X^*$. We then get, for all $t>0$:
		
		\begin{align*}
		\P(\Delta_i\in X^*_t\ |\ \Delta_i\notin X^*_0)&=\P(\Delta_i\in\tilde{X}^*_t\ |\ \Delta_i\notin \tilde{X}^*_0)\\
		&\leq\P(\underset{n\in\N}{\bigcap}\{\tilde{X}^*_t\cap\varphi_i(\dbc{Z^{i,n}_t+1,\infty})\neq\varnothing\}\ |\ \Delta_i\notin \tilde{X}^*_0)\\
		&\leq\P(\underset{n\in\N}{\bigcap}\{t\geq T_n\}\ |\ \Delta_i\notin\tilde{X^*_0})\\
		&=\P(\exists s\leq t:\;\tilde{X}^*_s=\{\Delta_i\}\ |\ \Delta_i\notin\tilde{X}^*_0)\\
		&=\P(\tilde{X}^*_t=\{\Delta_i\}\ |\ \Delta_i\notin\tilde{X}^*_0)
		\end{align*}
		
	and the latter is equal to $0$ from (\ref{absorbant}).
	
	\paragraph{}We show the second part with Lemmas \ref{couple-suffisante} and \ref{condition-suffisante}. Observe that since $X^*$ has the law of a process $\Lambda$-linked to $X$, which is positive recurrent, $X^*$ must reach $C_G$ in finite time almost surely. Thus we may assume without loss of generality that the process starts from this moment, so that $X^*_0$ contains at least one point of $C_G$.
	
	When the process $X^*$ reaches a state containing some $\Delta_i$ (necessarily after an explosion), the latter can exit $X^*$ only in the same time as all the rest of the branch $Q_i$ and the point $p_i$ of $C_G$ adjacent to $Q_i$. This event corresponds to a $B$ transition. The idea is therefore to show that the process can reach successively each of the $\Delta_i$ without losing any point of $C_G$ (so without losing the $\Delta_i$ already reached). To do so, we find a lower bound, uniform on the initial condition, of the probabilities to reach each $\Delta_i$ before a $B$ transition. We do this in two steps: first, we find a lower bound of the probability to reach $\varphi_i(0)$ by following a given path, then we couple the process with $Y$ and $Z^i$ as in Lemma \ref{couple-suffisante}.
	
	For every $i\in I$, we call $f_i:G^*\rightarrow G^*$ the function defined by:
		$$f_i(Q)=Q\cap (C_G\cup\{\varphi_i(0)\})$$
	A way to ensure the absence of $B$ transition on a time interval $[0,t]$ is to ask for $f_i(X^*)$ to be increasing (for set inclusion) on this interval.
	
	Let $t>0$, $i\in I,\; C\subset C_G\cup\{\varphi_i(0)\}$ such that $C\nsubseteq\{\varphi_i(0)\}$ (so $C$ contains at least one point of $C_G$), $Q\in G^*$ such that $f_i(Q)=C$ and $(q_0,\dots,q_{n_0})$ a path with $q_0\in C,\;q_1,\dots q_{n_0-1}\in C_G\setminus C$ and $q_{n_0}=\varphi_i(0)$. If $n_0=0$, that means that $\varphi_i(0)\in Q$. We write $C^n=C\cup \{q_1\}\cup\dots\cup \{q_n\},\; 0\leq n\leq n_0$. We also write $T_k$ for the $k$-th jumping time of $X^*,\;k\in\N$, with $T_0=0,\; (Y_k)_{k\in\N}$ the underlying jump chain: $X^*_{T_k}=Y_k$ for every $k\in\N$ and $T_{\phi(k)}$ the subsequence of jumping times of $f_i(X^*)$. The processes $(T_{\phi(k)})$ and $(f_i(Y_{\phi(k)}))$, expressing the evolution of $X^*$ on $C_G\cup\{\varphi_i(0)\}$, are not Markovian, but we can bound their transition probabilities using the process $X^*$. The first aim is to find a lower bound of $\P(f_i(X^*_t)=C^{n_0},\;f_i(X^*) \text{ increasing on }[0,t]\ |\ X^*_0=Q)$ that does not depend on $Q$. We start roughly:
	
		\begin{align}
		\nonumber&\P(f_i(X^*_t)=C^{n_0},\;f_i(X^*) \text{ increasing on }[0,t]\ |\ X^*_0=Q)\\
		\nonumber\geq&\P\left(\underset{k=1}{\overset{n_0}{\bigcap}}\{T_{\phi(k)}-T_{\phi(k-1)}\leq t/n_0\;,f_i(Y_{\phi(k)})=C^k,\;T_{\phi(n_0+1)}-T_{\phi(n_0)}>t\}\ |\ X^*_0=Q\right)\\
		=&\prod_{k=1}^{n_0}\P(\epsilon_k\in[0,t/n_0])\;p_k\times \P(\epsilon_{n_0+1}>t/n_0)\label{minor}
		\end{align}
		
	\noindent where for every $k\leq n_0+1$:
		$$\epsilon_k\sim\l(T_{\phi(k)}-T_{\phi(k-1)}\ |\ \underset{n=1}{\overset{k-1}{\bigcap}}\{f_i(Y_{\phi(n)})=C^n,\;T_{\phi(n)}-T_{\phi(n-1)}\in[0,t/n_0],\;X^*_0=Q\})$$
		$$p_k=\P(f_i(Y_{\phi(k)})=C^k\ |\ \underset{n=0}{\overset{k-1}{\bigcap}}\{f_i(Y_{\phi(n)})=C^{n},\;T_{\phi(n+1)}-T_{\phi(n)}\in[0,t/n_0],\;X^*_0=Q\})$$
	(observe that in the above conditional probability, $f_i(Y_0)=C^0$ is not necessary since it is implied by $X^*_0=Q$, and was left only by simplicity of notation).
	
	If at time $T_{\phi(k)}$ the process $X^*$ goes from some state $Q'$ to $Q''$ then we have $Q''=Q'\cup\{p\},\;p\in f_i(V(Q'))$ or $Q''\in\C(Q'\setminus\{p\}),\;p\in f_i(Q')$, and one easily checks from the coefficients of $\L^*$ that:
		$$\left(\underset{q\in V(f_i(G))}{\min}\mu(q)\right)\left(\underset{q\in V(C_G)\atop q'\in V(q)}{\min}\L_{q,q'}\right)\leq\L^*_{Q',Q''}\leq\dfrac{\underset{q\in V(C_G)}{\max}\;|\L_{q,q}|}{\underset{q'\in C_G}{\min}\mu(q')}$$
	We call $\alpha_1<\infty$ the r.h.s. and $\alpha_2>0$ the l.h.s. in the above inequality. One observes that for a given $Q'$, the number of possible $Q''$ of the form $Q'\cup\{p\}$ is upper bounded by $|C_G|$, and $|\underset{p\in f_i(Q')}{\bigcup}\C(Q'\setminus\{p\})|$ is upper bounded by $(|C_G|+1)^2$. Conditioning by $X^*_{T_{\phi(k)-1}}$, we get:
	
		\begin{align*}
		p_k&\geq\inf\{\P(X^*_{T_{\phi(k)}}=Q'\cup\{p\}\ |\ X^*_{T_{\phi(k)-1}}=Q'),\;f_i(Q')\neq\varnothing,\;p\in V(Q')\cap V(C_G)\}\\
		&=\inf\left\lbrace\dfrac{\L^*_{Q',Q'\cup\{p\}}}{\underset{Q''\in G^*\atop f_i(Q'')\neq f_i(Q')}{\sum}\L^*_{Q',Q''}},\;f_i(Q')\neq\varnothing,\;p\in V(Q')\cap V(C_G)\right\rbrace\\
		&\geq\dfrac{\alpha_2}{(|C_G|+(|C_G|+1)^2)\alpha_1}=:\alpha_0>0
		\end{align*}
	
	\noindent and, given $X^*_{T_{\phi(k-1)}}$ we can bound $T_{\phi(k)}-T_{\phi(k-1)}$ by two exponential variables $E_1$ (lower bound) and $E_2$ (upper bound), with respective parameters:
		$$\sup\left\lbrace\underset{Q''\in G^*\atop f_i(Q'')\neq f_i(Q')}{\sum}\L^*_{Q',Q''}:\;Q'\in G^*, f_i(Q')\neq\varnothing\right\rbrace\leq(|C_G|+(|C_G|+1)^2)\alpha_1$$
		$$\inf\left\lbrace\underset{Q''\in G^*\atop f_i(Q'')\neq f_i(Q')}{\sum}\L^*_{Q',Q''}:\;Q'\in G^*, f_i(Q')\neq\varnothing\right\rbrace\geq\alpha_2$$
	
	 Put in other words:
	 	$$\P(\epsilon_k\in[0,t/n_0])\geq\P(E_2\in[0,t/n_0])\geq 1-\exp(-\alpha_2t/n_0)$$
	 and:
	 	$$\P(\epsilon_{n_0+1}>t)\geq \P(E_1>t)\geq \exp(-(|C_G|+(|C_G|+1)^2)\alpha_1t)$$
	 Going back to (\ref{minor}) and using the fact that $n_0\leq|C_G|$, we see that we have found a lower bound $\varepsilon_t$ on the probability to reach $\varphi_i(0)$ in time $t$ before any $B$ transition, and that $\varepsilon_t$ does not depend on $i$, nor on the initial configuration $C$. Hence we have made the first step of the bounding.
	 
	 Starting from the state reached during this first step, we construct $(X^*,Y,Z^i)$ as in Lemma \ref{couple-suffisante}. The probability to reach a state containing $\Delta_i$ before a $B$ transition, in time $t$, is lower bounded by $\P(Z^i_{t}=\infty,Y>t\ |\ Z^i_0=\varphi_i(0))$ from the law of the couple $(X^*,Z^i)$, and this probability is positive since $Z^i$ and $Y$ are independent. Doing the two steps consecutively (each one in time $t/2$ for instance), we get:
		$$\P(\Delta_i\in X^*_t,f_i(X^*)\text{ increasing on }[0,t]\ |\ X^*_0,\; X^*_0\cap C_G\neq\varnothing)\geq\varepsilon^i_t$$
	and this constant only depends on $i$ and $t$. By Markovianity and homogeneity, it follows that:
		\begin{align*}
		&\P(\Delta_i\in X^*_{Nt},\;\forall i\in I\ |\ X^*_0\cap C_G\neq \varnothing)\\
		\geq&\prod_{i=1}^{N}\underset{Q\in G^*\atop Q\cap C_G\neq\varnothing}{\inf}\P(\Delta_i\in X^*_{it},f_i(X^*)\text{ increasing on }[(i-1)t,it]\ |\ X^*_{(i-1)t}=Q)\\
		\geq&\prod_{i=1}^{N}\underset{Q\in G^*\atop Q\cap C_G\neq\varnothing}{\inf}\P(\Delta_i\in X^*_{t},f_i(X^*)\text{ increasing on }[0,t]\ |\ X^*_0=Q)\\
		\geq&\prod_{i\in I}\varepsilon^i_t>0
		\end{align*}
	The set $G^*_N$ of states containing all $\Delta_i$ (defined in (\ref{set-G_i})) is then positive recurrent. It remains to conclude by observing that from some state $Q\in G^*_N$, we can go to $\bar{G}$ in finitely many jumps:
	if $q_1,\dots,q_{n_0}$ are such that $\{q_i,i\leq n_0\}=C_G\setminus Q$ and $q_i\in V(Q\cup\{q_1\}\cup\ldots\cup\{q_{i-1}\})$, then
		$$\P(X^*_{T_i}=Q\cup\{q_1\}\cup\ldots\cup\{q_i\},\forall i\leq n_0\ |\ X^*_0=Q)>\varepsilon$$
	where $\varepsilon$ is a positive constant independent from $Q$ and $n_0$.
	
	$\bar{G}$ is thus positive recurrent, and the process $X^*$ is absorbed in finite time almost surely.
	\end{proof}

After a last technical lemma, we will finally be able to prove Theorem \ref{thm}:
	
	\begin{Lem}\label{integrabilité}
	If $X^*$ is a stratified jump process with generator $\L^*$, such that $\E\left(\dfrac{1}{\mu(X^*_0)}\right)<+\infty$, then $\E\left(\dfrac{1}{\mu(X^*_t)}\right)<+\infty$ for all $t>0$.
	\end{Lem}

	\begin{proof}
	Let $A_n:=\underset{i\in I}{\bigcup}\varphi_i(\dbc{n,\infty})$ and
		$$\tau_n:=\inf\{t>0:\;X^*_t\subset A_n\}$$
	Observe that the sequence $(A_n)_{n\in\N}$ is decreasing and goes to $A_{\infty}=\{\Delta_i,i\in I\}$, whereas the sequence $(\tau_n)_{n\in\N}$ is increasing (starting from the first rank $n$ such that $\tau_n\neq0$) and goes to infinity (since the set $A_{\infty}$ cannot be reached in finite time). In particular, for all $t>0$:
		$$\P(X^*_{t\wedge\tau_n}\nsubseteq A_{n+1}\ |\ X^*_0\nsubseteq A_{n+1})=1$$
	In addition, if $X^*_0\nsubseteq A_{n+1}$, then for every $n\in\N$, we get:
		$$\dfrac{1}{\mu(X^*_{t\wedge\tau_n})}\leq\underset{q\in G\setminus A_{n+1}}{\sup}\dfrac{1}{\mu(q)}=:K_n$$
	We then set, for all $Q\in G^*$ and all $n\in\N$:
		$$g_n(Q)=\dfrac{1}{\mu(Q)}\wedge K_n$$
	These functions are continuous and bounded and for every $n$, $g_n$ coincides with $\dfrac{1}{\mu}$ on $\{Q\in G^*:\; Q\nsubseteq A_{n+1}\}$.
	Using the fact that 
		$$Q\nsubseteq A_n \Rightarrow Q'\nsubseteq A_{n+1},\;\forall Q'\in G^* \text{ such that }\L^*_{Q,Q'}>0$$
	we get, for every $n\in\N$, for all $Q\nsubseteq A_n$:
		\begin{align*}
		\L^*\dfrac{1}{\mu}(Q)&=\L^*g_{n}(Q)\\
		&=\sum_{Q'\neq Q}\L^*_{Q,Q'}\left(\dfrac{1}{\mu(Q')}-\dfrac{1}{\mu(Q)}\right)\\
		&=\sum_{q\notin Q}\left(\dfrac{\mu(Q\cup\{q\})}{\mu(Q)}\left(\dfrac{1}{\mu(Q\cup\{q\})}-\dfrac{1}{\mu(Q)}\right)\sum_{p\in Q}\L_{q,p}\right)\\
		&\quad+\sum_{p\in Q}\sum_{Q'\in\mathscr{C}(Q\setminus\{p\})}\left(\dfrac{\mu(Q')}{\mu(Q)}\left(\dfrac{1}{\mu(Q')}-\dfrac{1}{\mu(Q)}\right)\sum_{q\notin Q}\L_{p,q}\right)\\
		&=\sum_{q\notin Q}\dfrac{-\mu(q)}{\mu(Q)^2}\sum_{p\in Q}\L_{q,p}+\sum_{p\in Q}\sum_{Q'\in\mathscr{C}(Q\setminus\{p\})}\dfrac{\mu(Q)-\mu(Q')}{\mu(Q)^2}\sum_{q\notin Q}\L_{p,q}\\
		&=\sum_{p\in Q}\dfrac{|\mathscr{C}(Q\setminus\{p\})|-1}{\mu(Q)}\sum_{q\notin Q}\L_{p,q}
		\end{align*}
	The terms of this sum are null if $p\notin C_G$, and we deduce easily that this sum is upper bounded by a constant $K$ which does not depend on $n$. By Proposition \ref{martingale}, for every $n\in\N$ such that $X^*_0\nsubseteq A_{n+1}$ almost surely, the process:
		
		\begin{align*}
		M_t:=&\dfrac{1}{\mu(X^*_{t\wedge\tau_n})}-\dfrac{1}{\mu(X^*_0)}-\int_{0}^{t\wedge\tau_n}\L^*\dfrac{1}{\mu}(X_s^*)\mathrm{d}s\\
		=&g_{n}(X^*_{t\wedge\tau_n})-g_{n}(X^*_0)-\int_{0}^{t\wedge\tau_n}\L^*g_{n}(X_s^*)\mathrm{d}s
		\end{align*}
	
	\noindent is a martingale, hence:
		
		\begin{align*}
		\E\left(\dfrac{1}{\mu(X^*_{t\wedge\tau_n})}\right)&=\E\left(\dfrac{1}{\mu(X^*_0)}\right)+\E\left(\int_0^{t\wedge\tau_n}\L^*\dfrac{1}{\mu}(X^*_s)\mathrm{d}s\right)\\
		&\leq\E\left(\dfrac{1}{\mu(X^*_0)}\right)+K\E(t\wedge\tau_n)\\
		&\leq\E\left(\dfrac{1}{\mu(X^*_0)}\right)+Kt
		\end{align*}
		
	We conclude by dominated convergence that $\E\left(\dfrac{1}{\mu(X^*_{t})}\right)\leq \E\left(\dfrac{1}{\mu(X^*_0)}\right)+Kt$.
	\end{proof}

	\begin{proof}[Proof of Theorem \ref{thm}]
	If condition (\ref{borne-infini}) holds for every $i\in I$, then in a classical way we construct a strong stationary time from the explosion time of a strong stationary dual $X^*$ with generator $\L^*$, and it is finite.
	
	Conversely, let us assume that there exists $i\in I$ such that this condition does not hold. Let $\mu_0$ be with finite support. Let $X^*$ be a strong stationary dual of $X$, with generator $\L^*$. We use the separation function $\mathfrak{s}$, and the dominated convergence theorem together with Lemma \ref{integrabilité}:
	
		\begin{align*}
		\mathfrak{s}(t)&=\underset{q\in G}{\sup}\left(1-\dfrac{\mu_t(q)}{\mu(q)}\right)\\
		&=\underset{q\in G}{\sup}\ \E\left(1-\dfrac{\Lambda(X^*_t,q)}{\mu(q)}\right)\\
		&=1-\underset{q\in G}{\inf}\ \E\left(\dfrac{\delta_q(X^*_t)}{\mu(X^*_t)}\right)\\
		&\geq 1-\underset{q\rightarrow\infty}{\lim}\ \E\left(\dfrac{\delta_{\varphi_i(q)}(X^*_t)}{\mu(X^*_t)}\right)\\
		&=1-\ \E\left(\underset{q\rightarrow\infty}{\lim}\ \dfrac{\delta_{\varphi_i(q)}(X^*_t)}{\mu(X^*_t)}\right)\\
		&=1
		\end{align*}
	
	Now, we know that if $T$ is a strong stationary time of $X$, then for all $t>0,\; \P(T>t)\geq\mathfrak{s}(t)$. So $T=\infty$ almost surely.
	\end{proof}

\textbf{Acknowledgements}:
I thank my Ph.D advisor L. Miclo for introducing this problem to me and for fruitful discussions, and Pan Zhao for pointing out some imprecisions in the previous version.

\bibliographystyle{plain} 
\bibliography{biblio}

\end{document}